\documentclass{article}
\usepackage{bm}
\usepackage{tikz,mathpazo}
\usepackage{pgf}
\usepackage[top=2.5cm, bottom=2.5cm, left=3cm, right=3cm]{geometry}   
\usepackage{indentfirst}
\usepackage{graphicx}
\usepackage{graphics}
\usepackage[toc,page,title,titletoc,header]{appendix}
\usepackage{bm}
\usepackage{listings}  
\usepackage{amsmath}
\usepackage{setspace} 
\usepackage{indentfirst}
\usepackage{caption}
\usepackage{multirow} 
\usepackage{lipsum,multicol}
\usepackage{pdfpages}
\usepackage{float}
\usepackage{amsthm}
\usepackage{pdfpages}
\usepackage{url}
\usepackage{colortbl}
\usepackage{subfigure}
\usepackage{epsfig}
\usepackage{epstopdf}
\usetikzlibrary{shapes.geometric, arrows}
\usepackage{fancyhdr}
\usepackage{abstract}
\usepackage{tikz,mathpazo}
\usetikzlibrary{shapes.geometric, arrows}

\usepackage{amssymb}
\usepackage{latexsym}
\usepackage{verbatim}
\usepackage[numbers]{natbib} 
\usepackage{booktabs}

\usepackage{tikz}

\newcommand{\inner}[1]{\left\langle #1 \right\rangle}
\newcommand{\norm}[1]{\left\Vert #1\right\Vert}
\newcommand{\bb}[1]{\mathbb{#1}}

\newcommand{\conv}[0]{\mathrm{conv}\,}
\newcommand{\cl}[0]{\mathrm{cl}\,} 

\newcommand{\ca}[1]{\mathcal{#1}}

\newcommand{\M}[0]{\mathcal{M}}

\newcommand{\tp}{^\top}

\newcommand{\A}{\ca{A}}

\newcommand{\D}{\ca{D}}
\newcommand{\Dh}{\ca{D}_h}
\newcommand{\hDh}{\hat{\ca{D}}_h}
\newcommand{\hDp}{\hat{\ca{D}}_p}

\newcommand{\hDs}{\hat{\ca{D}}_s}

\newcommand{\etak}{{\eta_{k} }}

\newcommand{\xk}{{x_k}}
\newcommand{\xkp}{{x_{k+1}}}
\newcommand{\yk}{{y_k}}
\newcommand{\ykp}{{y_{k+1}}}

\newcommand{\ysol}{{y^\star}}

\newcommand{\Rn}{\mathbb{R}^n}
\newcommand{\Rp}{\mathbb{R}^p}

\newcommand{\JAx}{J_{\mathcal{A},x}}
\newcommand{\JAy}{J_{\mathcal{A},y}}
\newcommand{\X}{\mathcal{X}}

\newtheorem{theo}{Theorem}[section]
\newtheorem{lem}[theo]{Lemma}
\newtheorem{prop}[theo]{Proposition}

\newtheorem{coro}[theo]{Corollary}

\newtheorem{defin}[theo]{Definition}
\newtheorem{rmk}[theo]{Remark}
\newtheorem{assumpt}[theo]{Assumption}

\usepackage{caption}
\usepackage{algorithm}
\usepackage{algpseudocode}
\usepackage{longtable}
\usepackage{appendix}
\usepackage{enumitem}

\numberwithin{equation}{section}

\title{An Improved Unconstrained Approach for Bilevel Optimization}

\author{
	{
		Xiaoyin Hu, ~
		Nachuan Xiao,~
		Xin Liu, ~
		and ~ Kim-Chuan Toh
	}
}

\begin{document}
	\maketitle
	
	\begin{abstract}
		In this paper, we focus on the nonconvex-strongly-convex bilevel optimization problem (BLO). In this BLO, the objective function of the upper-level problem is nonconvex and possibly nonsmooth, and the lower-level problem is smooth and strongly convex with respect to the underlying variable $y$. We show that the feasible region of BLO is a Riemannian manifold. Then we transform BLO to its corresponding unconstrained constraint dissolving problem (CDB),  whose objective function is explicitly formulated from the objective functions in BLO. 
		We prove that BLO is equivalent to the unconstrained optimization problem CDB. Therefore, various efficient unconstrained approaches, together with their theoretical results, can be directly applied to BLO through CDB.  We propose a unified framework for developing subgradient-based methods for CDB. Remarkably, we show that several existing efficient algorithms can fit the unified framework and be interpreted as descent algorithms for CDB.  These examples further demonstrate the great potential of our proposed approach.
	\end{abstract}

\section{Introduction}
In this paper, we focus on the following nonconvex-strongly-convex bilevel optimization problem
\begin{equation}
	\tag{BLO}
	\label{Prob_Ori}
	\begin{aligned}
		\min_{x \in \Rn, y \in \Rp} \quad &f(x, y) & \text{(upper-level problem)}\\
		\text{s. t.} \quad &y = \mathop{\arg\min}_{y \in \bb{R}^p}~ g(x,y), & \text{(lower-level problem)}
	\end{aligned}
\end{equation}
where the functions $f$ and $g$ satisfy the following blanket assumptions,

\begin{assumpt}{\bf Blanket assumptions}
	\label{Assumption_1}
	
	\begin{enumerate}
		\item $f$ is possibly nonsmooth and $M_f$-Lipschitz continuous over $\Rn\times \Rp$.
		\item The function $g(x,y)$ is twice differentiable and $\mu$-strongly convex with respect to $y$ for any fixed $x$, i.e. $\nabla_{yy}^2 g(x, y) \succeq \mu I_p$ holds for any $(x, y) \in \Rn\times \Rp$. 
		\item The gradient $\nabla g(x,y)$ is $L_g$-Lipschitz continuous. 
		\item The Hessian matrices $\nabla_{yy}^2 g(x, y)$ and $\nabla^2_{xy} g(x,y)$ are $Q_g$-Lipschitz continuous. 
		\item $\nabla_{yy}^2 g(x, y)$ is continuously differentiable over $\Rn\times \Rp$. 
	\end{enumerate}
\end{assumpt} 
Problem \ref{Prob_Ori} has attracted a lot of attention in the current era of big data and artificial intelligence due to its close connection with 
various real-world applications, including reinforcement learning \cite{konda1999actor}, hyperparameter optimization \cite{tappen2008logistic,hutter2011sequential,franceschi2017forward,lorraine2020optimizing}, and meta learning \cite{finn2017model,samuel2009learning}. Interested reader can refer to several survey papers \cite{colson2007overview,liu2021investigating} and the references therein for details. 

The blanket assumption \ref{Assumption_1} is commonly assumed in a great number of existing works. 
In particular, Assumptions \ref{Assumption_1} is satisfied in the applications discussed in \cite{domke2012generic,maclaurin2015gradient,pedregosa2016hyperparameter,franceschi2018bilevel,ghadimi2018approximation,liao2018reviving,grazzi2020iteration,ji2021bilevel}.  It should be noted that although we assume the Lipschitz smoothness of $\nabla_{yy}^2g(x, y)$, it is only necessary in
the theoretical analysis, and we do not involve the computation of any third-order derivatives in our proposed methods throughout this paper.

\subsection{Existing works}

Recently, nonconvex-strongly-convex bilevel optimization problems with Lipschitz smooth objective functions have been extensively studied.
For any given $x \in \Rn$, we denote $\ysol(x)$ as the unique minimizer of the lower-level problem, i.e., $\ysol(x) := \mathop{\arg\min}_{y \in \bb{R}^p}~ g(x,y)$.  Since the lower-level problem of \ref{Prob_Ori} is assumed to be strongly convex with respect to $y$, $\ysol(x)$ is 
differentiable with respect to $x$ by the implicit function theorem. Therefore, 
\ref{Prob_Ori} is equivalent to the following unconstrained optimization problem that only involves the $x$-variable,
\begin{equation}
	\label{Eq_UNCon}
	\min_{x \in \Rn} \quad \Phi(x) := f(x, \ysol(x)).
\end{equation} 
Various existing efficient approaches are developed based on solving the unconstrained optimization problem \eqref{Eq_UNCon}. 
However,  $\Phi(x)$ is implicitly formulated since the solution to the lower-level problem usually does not have a closed-form expression \cite{hong2020two}. Therefore,  it is usually intractable to compute the exact function value and derivatives of $\Phi(x)$.

Some of the existing approaches \cite{domke2012generic,pedregosa2016hyperparameter,ghadimi2018approximation,grazzi2020iteration,ji2021bilevel}, referred to as {\it double-loop} approaches, are developed by introducing inner loops in each iteration to obtain an approximated estimation for $\ysol(x)$. Then these approaches inexactly evaluate $\nabla \Phi(x)$ through the approximated solution for the lower-level problem and chain rule.  Although their theoretical properties are simple to analyze, these algorithms may suffer from  poor performance as one has to take multiple steps in the inner loop to solve the lower-level problem to a desired accuracy \cite{khanduri2021near}. It is usually challenging to balance the computational cost of the inner loops and the overall performance of these algorithms.

Furthermore, several {\it single-loop} approaches \cite{chen2021single,hong2020two,khanduri2021near} are proposed to minimize $\Phi(x)$ by updating the  $x$- and $y$-variables simultaneously, hence avoiding inner loops for an approximated solution of the lower-level problem. 
In each iteration, these single-loop approaches update the $x$-variable by taking an approximated gradient descent step to $\Phi(x)$, while the $y$-variable is updated to track $\ysol(x)$ by taking a descent step for the lower-level problem \cite{hong2020two,khanduri2021near} or other specifically designed schemes \cite{chen2021single}.  
Although prior arts \cite{hong2020two,khanduri2021near,chen2021single,yang2021provably} use $\Phi(x)$ as the merit function in their theoretical analysis, these existing single-loop approaches cannot be simply interpreted as approximated gradient descent methods to minimize $\Phi(x)$. Therefore, establishing the related theoretical analysis for these  approaches becomes more complicated and challenging in these existing works.

Though solving \ref{Prob_Ori} with smooth objective functions has been intensively studied, how to solve \ref{Prob_Ori} with a nonsmooth upper-level objective function is relatively less explored. 
Due to the implicit formulation of $\Phi(x)$, existing single-loop and double-loop approaches have to approximately solve the lower-level problems and evaluate $\nabla \Phi(x)$ inexactly. 
Therefore, without the assumption on the Lipschitz smoothness of $f$, the above-mentioned approaches have no theoretical guarantee. On the other hand, computing the exact subdifferential of $\Phi$ requires the exact solution to the lower-level problem, which is usually expensive to achieve in practice. As a result, it is challenging to develop algorithms for \ref{Prob_Ori} based on $\Phi(x)$.

Apart from those existing approaches developed for minimizing $\Phi(x)$ over $\Rn$, several other existing approaches \cite{hansen1992new,moore2010bilevel} reshape \ref{Prob_Ori} as the following single-level optimization problem with equality constraints \cite{ye2006constraint}
\begin{equation}
	\label{Eq_Con}
	\begin{aligned}
		\min_{x \in \bb{R}^n, y \in \Rp} \quad &f(x, y)\\
		\text{s. t.} \quad &\nabla_y g(x,y) = 0. 
	\end{aligned}
\end{equation}
Then  \eqref{Eq_Con} can be solved by employing existing approaches for constrained optimization, including polynomial optimization approach (when the functions involved are polynomials) \cite{nie2021lagrange},  sequential quadratic programming methods \cite{curtis2012a,xu2015smoothing}, etc.  However, these approaches treat \ref{Prob_Ori} as a constrained optimization problem with $p$ equality constraints, hence they are usually not as efficient as those aforementioned single-loop and double-loop approaches in practice \cite{hong2020two}.

\subsection{Motivation}
Our motivation in this paper comes from the {\it constraint dissolving} approaches \cite{xiao2022constraint} for Riemannian optimization. Let $\M$ be the feasible region of \eqref{Eq_Con}, i.e. 
\begin{equation}
	\label{Eq_M}
	\begin{aligned}
		\M := \{ (x, y) \in \Rn\times \Rp: \nabla_y g(x, y) = 0 \}.
	\end{aligned}
\end{equation}
As $g(x, y)$ is twice-order differentiable and strongly convex with respect to $y$, the constraints $\nabla_y g(x, y) = 0$ satisfy linear independent constraint qualification (LICQ) for any   $(x,y) \in \M$. Therefore, the implicit function theorem ensures that $\M$ is a Riemannian manifold embedded in $\Rn\times \Rp$ \cite{jones2004calculus}. Although various Riemannian optimization approaches are developed in recent years \cite{Absil2009optimization,boumal2020introduction,hu2020brief}, the required geometrical materials of the manifold $\M$ are usually expensive to compute. For example, computing the retraction of $\M$ can be regraded as computing a projection from the tangent space to $\M$, which is as expensive as solving the lower-level subproblem exactly. To our best knowledge, there is no efficient Riemannian optimization approach developed for solving \ref{Prob_Ori}.  

When $f$ is assumed to be Lipschitz smooth over $\Rn \times \Rp$, \cite{xiao2022constraint} proposes a general framework for developing the constraint dissolving function for \ref{Prob_Ori}, which takes the form as 
\begin{equation}
	\label{Eq_CDF}
	f(\tilde{\A}(x, y)) + \frac{\beta}{2} \norm{\nabla_y g(x, y)}^2.
\end{equation}
Here,  the mapping $\tilde{\A}: \Rn\times \Rp \to \Rn\times \Rp$ satisfying the following assumptions is called 
	the constraint dissolving mapping.
	\begin{assumpt}
		\label{Assumption_constraint_dissolving}
		\begin{itemize}
			\item $\tilde{\A}$ is locally Lipschitz continuous over $\Rn\times \Rp$.
			\item $\tilde{\A}(x, y) = (x,y)$ for any $(x,y) \in \M$.
			\item The Jacobian of $(\nabla_y g) \circ \tilde{\A}$ equals to $0$ for any $(x, y) \in \M$. 
		\end{itemize}
\end{assumpt}
As illustrated in \cite[Lemma 3.3]{xiao2022constraint}, any constraint dissolving mapping $\tilde{\A}$ will drive any $(x, y) \in \Rn \times \Rp$ closer to the feasible region $\M$ 
	with the feasibility violation locally quadratically converges to zero. This property
	plays a crucial role in establishing the equivalence between the original bilevel optimization problem (BLO)
	and minimizing the constraint dissolving function \eqref{Eq_CDF}. 
	The detailed proof can be found at \cite[Section 3]{xiao2022constraint}.

Moreover,  \cite{xiao2022constraint} provides some practical schemes for constructing the constraint dissolving mapping, see \cite[Section 4.1]{xiao2022constraint} for instances. However, \cite{xiao2022constraint} focuses on smooth optimization over the Riemannian manifold. Existing constraint dissolving approaches for nonsmooth optimization are only developed for special manifolds \cite{hu2022constraint}. Furthermore, the equivalence established in \cite{xiao2022constraint} only holds in a neighborhood of the feasible region $\M$.
	For general nonsmooth cases, how to choose an appropriate constraint dissolving operator $\tilde{\A}$ for \eqref{Eq_CDF} and establish the equivalence between \eqref{Prob_Ori} and \eqref{Eq_CDF} over $\Rn\times \Rp$ rather than a neighborhood of $\M$ remain to be studied.

\subsection{Contributions}
In this paper, 
	we consider the mapping $(x, y) \mapsto (x, \A(x, y))$, where $\A$ is defined by 
	\begin{equation}\label{CDB-A}
		\A(x,y) := y - \left( \nabla_{yy}^2g(x,y) \right)^{-1} \nabla_y g(x,y)
	\end{equation}
	as a special choice of $\tilde{\A}$. 
	Substituting this $\tilde{\A}$ into \eqref{Eq_CDF}, we obtain a 
	constraint dissolving function for bilevel optimization (\ref{Prob_Pen})
	\begin{equation}
		\tag{CDB}
		\label{Prob_Pen}
		\begin{aligned}
			\min_{x \in \bb{R}^{n}, y \in \bb{R}^p}\quad h(x, y) :=f\left( x,  \A(x,y) \right) + \frac{\beta}{2} \norm{\nabla_y g(x,y)}^2.
		\end{aligned}
	\end{equation}

We prove that such an $\tilde{\A}$ satisfies Assumption 1.2 and hence is a constraint dissolving mapping \cite{xiao2022constraint}. Clearly,  $h$ can be explicitly formulated from $f$ and the derivatives of $g$. 
Under mild conditions, we prove that \ref{Prob_Ori} and \ref{Prob_Pen} have the same stationary points over $\Rn\times \Rp$ from the perspective of both the Clarke subdifferential and the conservative field \cite{bolte2021conservative}. As a result,  the bilevel optimization problem \ref{Prob_Ori} is equivalent to the unconstrained optimization problem \ref{Prob_Pen}, and various optimization approaches for unconstrained nonsmooth optimization can be  directly implemented to solve \ref{Prob_Ori} through \ref{Prob_Pen}.

We propose a unified framework for developing subgradient-based methods to solve \ref{Prob_Pen} and prove their global convergence. We provide several illustrative examples on how to develop single-loop subgradient-based methods and how to establish their convergence properties from the proposed framework. Moreover, we can interpret the updating schemes in the deterministic versions of several existing single-loop algorithms \cite{hong2020two,chen2021single,khanduri2021near} as  approximated gradient-descent steps for \ref{Prob_Pen}. Therefore, we provide a clear explanation for the updating schemes in these existing algorithms,  extend these algorithms to nonsmooth cases and prove their convergence properties based on our proposed framework. These examples further highlight the significant advantages and great potentials of \ref{Prob_Pen}.

\section{Preliminaries}
\subsection{Basic notations}
Let $\inner{\cdot, \cdot}$ be the standard inner product  and $\norm{\cdot}$ be the $\ell_2$-norm of a vector or an operator. $\bb{B}_{\delta}(x,y):= \{ (\tilde{x}, \tilde{y}) \in \Rn \times \Rp: \norm{\tilde{x} - x}^2 + \norm{\tilde{y} - y}^2 \leq \delta^2 \}$ refers to the ball
	centered at $(x,y)$ with radius $\delta$. Moreover, for a given set $\X$, $\mathrm{dist}(x, \ca{X})$ denotes the distance between $x$ and a set $\ca{X}$, i.e. $\mathrm{dist}(x, \ca{X}) := \mathop{\arg\min}_{y \in \ca{X}} ~\norm{x-y}$,  $\cl \X$ denotes the closure of $\X$ and $\conv \X$ denotes the convex hull of $\X$. 
For any differentiable function $g: \Rn \times \Rp \to \bb{R}$, let $\nabla_x g$ and $\nabla_y g$ be the partial derivatives of $g$ with respect to $x$ and $y$, respectively. Moreover, $\nabla_{xy}^2 g(x, y)$ and  $\nabla_{yy}^2 g(x, y)$ denotes the partial Jacobian of $\nabla_y g(x, y)$ with respect to variable $x$ and $y$, respectively. More precisely, 
\begin{equation*}
	\footnotesize
	\nabla_{xy}^2 g(x, y) := \left[  \begin{matrix}
		\frac{\partial^2 g(x, y) }{\partial x_1\partial y_1} & \cdots & \frac{\partial^2 g(x, y) }{\partial x_1\partial y_p} \\
		\vdots & \ddots & \vdots \\
		\frac{\partial^2 g(x, y) }{\partial x_n\partial y_1} & \cdots & \frac{\partial^2 g(x, y) }{\partial x_n\partial y_p} \\
	\end{matrix}\right] \in \bb{R}^{n\times p}, 
	\quad 
	\nabla_{yy}^2 g(x, y) := \left[  \begin{matrix}
		\frac{\partial^2 g(x, y) }{\partial y_1\partial y_1} & \cdots & \frac{\partial^2 g(x, y) }{\partial y_1\partial y_p} \\
		\vdots & \ddots & \vdots \\
		\frac{\partial^2 g(x, y) }{\partial y_p\partial y_1} & \cdots & \frac{\partial^2 g(x, y) }{\partial y_p\partial y_p} \\
	\end{matrix}\right] \in \bb{R}^{p\times p}, 
\end{equation*} 
and   $\nabla_{yx}^2 g(x, y)$ is the transpose of $\nabla_{xy}^2 g(x, y)$.
Furthermore, $\nabla_{xyy}^3 g(x, y)$ is the partial derivative of $\nabla_{xy}^2 g(x, y)$ with respect to variable $y$, which is expressed as the linear mapping from $\Rp$ to $\bb{R}^{n\times p}$ by $\nabla_{xyy}^3 g(x, y) [d] := \lim\limits_{t\to 0} ~\frac{1}{t} \left(\nabla_{xy}^2g(x, y+td) -  \nabla_{xy}^2g(x, y) \right)$. 
Similarly, $\nabla_{yyy}^3 g(x, y)$ is the partial derivative of $\nabla_{yy}^2 g(x, y)$ with respect to variable $y$, which is expressed as the linear mapping from $\Rp$ to $\bb{R}^{p\times p}$ by $\nabla_{yyy}^3 g(x, y) [d] := \lim\limits_{t\to 0} ~\frac{1}{t} \left(\nabla_{yy}^2 g(x, y+td) -  \nabla_{yy}^2g(x, y) \right)$. 
Under Assumption \ref{Assumption_1}, it is easy to verify that the inequalities $\norm{\nabla_{yyx}^3 g(x, y)[d_x]} \leq Q_g\norm{d_x}$ and $\norm{\nabla_{yyy}^3 g(x, y)[d_y]} \leq Q_g\norm{d_y}$ hold for all $(x,y)\in \Rn\times \Rp$ and $(d_x,d_y)\in  \Rn\times \Rp$.

\subsection{Clarke subdifferential}

\begin{defin}
	\label{Defin_Subdifferential}
	For any given locally Lipschitz continuous function $f: \Rn\times \Rp \to \bb{R}$ and any $(x, y) \in \Rn \times \Rp$, the generalized directional derivative of $f$ at $(x,y)$  in the direction $(d_x, d_y) \in \Rn \times \Rp$, denoted by $f^\circ(x,y; d_x, d_y)$, is defined as 
	\begin{equation*}
		f^\circ(x,y; d_x, d_y) := \mathop{\lim\sup}_{(\tilde{x}, \tilde{y})\to (x, y), ~t \downarrow  0}~ \frac{f(\tilde{x} + td_x, \tilde{y} + t d_y) - f(\tilde{x}, \tilde{y})}{t}. 
	\end{equation*}
	Then the generalized gradient or the Clarke subdifferential of $f$ at $(x, y)$, denoted by $\partial f(x, y)$, is defined as 
	\begin{equation*}
		\begin{aligned}
			\partial f(x, y) := &\left\{ (w_x, w_y) \in \Rn\times \Rp : \right. \\
			&\left.\inner{w_x, d_x} + \inner{w_y, d_y} \leq f^\circ (x, y; d_x, d_y), \text{ for all } (d_x, d_y) \in \Rn\times \Rp\right\}.
		\end{aligned}
	\end{equation*}
\end{defin}

\begin{rmk}
	For any locally Lipschitz continuous function $f: \Rn\times \Rp \to \bb{R}$, its Clarke subdifferential is compact and convex for any $(x, y) \in \Rn \times \Rp$. Moreover, the mapping $(x,y) \mapsto \partial f(x, y)$ is outer-semicontinuous over $\Rn \times \Rp$ \cite{clarke1990optimization}. 
\end{rmk}

\begin{defin}
	\label{Defin_Clarke_regular}
	We say that $f$ is (Clarke) regular at $(x, y) \in \Rn \times \Rp$ if for every direction $(d_x, d_y)\in \Rn \times \Rp$, the one-sided directional derivative 
	\begin{equation*}
		f^\star(x,y;d_x, d_y) := \lim_{t \downarrow 0} \frac{f(x + td_x, y+td_y) - f(x, y)}{t} 
	\end{equation*}
	exists and $f^\star(x,y; d_x, d_y) = f^\circ(x,y;d_x, d_y)$.
\end{defin}

	\begin{defin}
		\label{Defin_Eps_Subdifferential}
		For any given locally Lipschitz continuous function $f: \Rn\times \Rp \to \bb{R}$ and any $(x, y) \in \Rn \times \Rp$, the $\delta$-Goldstein subdifferential of $f$ at $(x, y)$  is defined as 
		\begin{equation*}
			\partial_{\delta} f(x, y) = \cl \conv\left(\cup_{(\tilde{x}, \tilde{y}) \in \bb{B}_{\delta}(x, y)} \partial f(\tilde{x}, \tilde{y})\right). 
		\end{equation*}
		
	\end{defin}
	
	The following proposition present some basic properties of $\delta$-Goldstein subdifferential, which are mainly from the upper-semicontinuity of $\partial f$, as illustrated in \cite[Theorem 3.1]{burke2002approximating} and \cite[Lemma 7]{zhang2020complexity}. 
	
	\begin{prop}
		\label{Prop_Eps_Subdifferential}
		For any given locally Lipschitz continuous function $f: \Rn\times \Rp \to \bb{R}$ and any $(x, y) \in \Rn \times \Rp$, it holds that 
		\begin{equation*}
			\lim_{\delta \to 0} \partial_{\delta} f(x, y) = \partial f(x, y). 
		\end{equation*}
		
	\end{prop}

\subsection{Conservative field}
\label{Subsection_conservative_field}
In this subsection, we introduce the concept of conservative field, which generalizes Clarke subdifferential for a broad class of nonsmooth functions. For simplicity, we provide a self-contained description and highlight some essential ingredients for our theoretical analysis. Interested readers can refer to several recent papers \cite{bolte2021conservative,castera2021inertial} for more details.

\begin{defin}
	A set-valued mapping $\ca{D}: \bb{R}^m \rightrightarrows \bb{R}^s$ is a mapping from $\bb{R}^m$ to a collection of subsets of $\bb{R}^s$. $\D$ is said to have closed graph if the graph of $\ca{D}$, defined by
	\begin{equation*}
		\mathrm{graph}(\D) := \left\{ (w,z) \in \bb{R}^m \times \bb{R}^s: w \in \bb{R}^m, z \in \D(w) \right\},
	\end{equation*}
	is a closed set.  
\end{defin}

\begin{defin}
	An absolutely continuous curve is a continuous mapping $\gamma: \bb{R} \to \Rn\times \Rp $ whose derivative $\gamma'$ exists almost everywhere in $\bb{R}$ and $\gamma(t) - \gamma(0)$ equals to the Lebesgue integral of $\gamma'$ between $0$ and $t$ for all $t \in \bb{R}_+$, i.e.,
	\begin{equation*}
		\gamma(t) = \gamma(0) + \int_{0}^t \gamma'(\tau) \mathrm{d} \tau, \qquad \text{for all $t \in \bb{R}_+$}.
	\end{equation*}
\end{defin}

With the concept of absolutely continuous curve, we can present the definition of a conservative set-valued field. 
\begin{defin}
	\label{Defin_conservative_field}
	Let $\ca{D}$ be a set-valued mapping from $\Rn\times \Rp  $ to subsets of $\Rn\times \Rp  $. Then we call $\ca{D}$ as a conservative field whenever it has closed graph, nonempty compact values, and for any absolutely continuous curve $\gamma: [0,1] \to \Rn\times \Rp $ satisfying $\gamma(0) = \gamma(1)$, we have
	\begin{equation}
		\label{Eq_Defin_Conservative_mappping}
		\int_{0}^1 \max_{v \in \ca{D}(\gamma(t)) } \inner{\gamma'(t), v} \mathrm{d}t = 0, 
	\end{equation}
	where the integral is understood in the Lebesgue sense. 
\end{defin}

\begin{rmk}
	When the set-valued mapping $\ca{D}$ has compact values and closed graph, then the  mapping $t \mapsto \max_{v \in \ca{D}(\gamma(t)) } \inner{\gamma'(t), v}$ 
	is Lebesgue measurable \cite[Lemma 1]{bolte2021conservative}. Therefore, the path integral in \eqref{Eq_Defin_Conservative_mappping} is well-defined. Furthermore, the equation \eqref{Eq_Defin_Conservative_mappping}   can be replaced by $\int_{0}^1 \min_{v \in \ca{D}(\gamma(t)) } \inner{\gamma'(t), v} \mathrm{d}t = 0$. 
\end{rmk}

\begin{defin}
	\label{Defin_conservative_field_path_int}
	Let $\ca{D}$ be a conservative field in $\Rn\times \Rp$. Then with any given $(x_0, y_0) \in \Rn\times \Rp$, we can define a function through
	\begin{equation}
		\label{Eq_Defin_CF}
		\begin{aligned}
			f(x, y) = &{} f(x_0, y_0) + \int_{0}^1 \max_{v \in \ca{D}(\gamma(t)) } \inner{\gamma'(t), v} \mathrm{d}t
			= f(x_0, y_0) + \int_{0}^1 \min_{v \in \ca{D}(\gamma(t)) } \inner{\gamma'(t), v} \mathrm{d}t
		\end{aligned}
	\end{equation} 
	for any absolutely continuous curve $\gamma$ that satisfies $\gamma(0) = (x_0, y_0)$ and $\gamma(1) = (x, y)$.  Then $f$ is called a potential function for $\ca{D}$, and we also say $\ca{D}$ admits $f$ as its potential function, or that $\ca{D}$ is a conservative field for $f$. 
\end{defin}

It is worth mentioning that any conservative field defines a unique potential function up to a constant, since the value of the integral does not depend on the selection of the path in \eqref{Eq_Defin_CF}.  Moreover, for any $f$ that is a potential function for some conservative field $\ca{D}$, $\partial f$ is a conservative field that admits $f$ as its potential function, and $\partial f(x, y) \subseteq \mathrm{conv}(\ca{D}(x, y))$ holds for any $(x, y) \in \Rn\times \Rp$ \cite[Corollary 1]{bolte2021conservative}.

As a result, for any $(x, y) \in \Rn\times \Rp$ that is a first-order stationary point of $f$, then it holds that $0 \in \partial f(x, y) \subseteq \conv(\ca{D}(x, y))$. Thus the stationarity of the potential function $f$ can be characterized by its corresponding conservative field $\ca{D}$ as illustrated in the following definition. 
\begin{defin}
	\label{Defin_D_critical}
	Given a fixed conservative field $\ca{D}: \Rn \times \Rp \rightrightarrows \Rn \times \Rp$ that  admits $f$ as a potential function, then we say $(x, y)$ is a $\ca{D}$-stationary point for $f$ if $0 \in \ca{D}(x, y) $. 
\end{defin}

%
%
%

Similar to the definition on conservative field, we present the definition on conservative mapping as follows. 
\begin{defin}
	\label{Defin_conservative_mapping}
	Let $F: \bb{R}^d \to \bb{R}^m$ be a locally Lipschitz function. $J_F : \bb{R}^d \rightrightarrows \bb{R}^{m\times d}$ is called a conservative mapping for $F$, if for any absolutely continuous curve $\gamma: [0,1] \to \bb{R}^d$, the function $t\mapsto F(\gamma(t))$ satisfies
	\begin{equation*}
		\frac{\mathrm{d} (F\circ\gamma)}{\mathrm{d}t} (t) = V\gamma'(t), \quad \text{for all } V \in J_F(\gamma(t)) \text{ and a.e. $t \in [0,1]$}.
	\end{equation*}
\end{defin}
When we choose $m = 1$ in Definition \ref{Defin_conservative_mapping}, the definition on conservative mapping is equivalent to the definition on conservative field in Definition \ref{Defin_conservative_field}, as illustrated in  \cite[Remark 7]{bolte2021conservative}. The following propositions illustrate that the chain rule and sum rule hold for conservative fields.

\begin{prop}[Lemma 7 in \cite{bolte2021conservative}]
	\label{Prop_Conservative_Chain_rule}
	Let $F_1: \bb{R}^d \to \bb{R}^m$ and $F_2 : \bb{R}^m \to \bb{R}^s$ be locally Lipschitz continuous mappings,  $J_{F_1}: \bb{R}^d \to \bb{R}^{m\times d}$ and $J_{F_2}: \bb{R}^d \to \bb{R}^{s\times m}$ be their associated conservative mappings. Then the mapping $x\mapsto J_{F_2}(F_1(x)) J_{F_1}(x)$ is a conservative mapping for $F_2 \circ F_1$. 
\end{prop}

\begin{prop}[Corollary 4 in \cite{bolte2021conservative}]
	\label{Prop_Conservative_Sum_rule}
	Let $f_1, ..., f_n$ be locally Lipschitz continuous functions for the conservative fields $\ca{D}_{f_1}, ..., \ca{D}_{f_n}$, respectively. Then $f = \sum_{i= 1}^n f_i$ is a potential function for $\ca{D}_f = \sum_{i= 1}^n \ca{D}_{f_i}$. 
\end{prop}

\subsection{Additional assumption and stationarity}

In this subsection, we present the basic assumptions on \ref{Prob_Ori} as well as the definition of its stationarity. In the rest of this paper, we assume the objective function $f$ to be a potential function for a certain conservative field.

\begin{assumpt}
	\label{Assumption_2}
	$f$ is a potential function of a conservative set-valued field $\ca{D}_f: \Rn\times \Rp  \rightrightarrows \Rn\times \Rp $, which has convex values, and satisfies
	$$\sup\limits_{x \in \Rn,~ y\in \Rp, ~\xi \in \ca{D}_f(x, y)}~ \norm{\xi} \leq M_f,$$
	for some constant $M_f >0$.
\end{assumpt}

The following remark illustrates that Assumption \ref{Assumption_2} is general enough to cover most applications of \ref{Prob_Ori}. 

\begin{rmk}
	\label{Rmk_stratifiable_functions}
	It is worth mentioning that any Clarke regular function is a potential function for some conservative fields \cite{davis2020stochastic}. However, the Clarke regularity is too restrictive in practice, which excludes some important applications of \ref{Prob_Ori}, in particular, training the neural network built from nonsmooth activation functions. 
	
	To this end, \cite{davis2020stochastic} reviews the concept of Whitney stratifiable functions, and prove that any locally Lipschitz function $f:\Rn \times \Rp \to \bb{R}$ that is Whitney $C^1$-stratifiable is a potential function for $\partial f$ in \cite[Theorem 5.8]{davis2020stochastic}. 
	Whitney stratifiable functions are general enough to cover several important classes of functions, including semi-algebraic functions, and semi-analytic functions.
	
	Additionally, several recent works \cite{bolte2021conservative,castera2021inertial,bolte2021nonsmooth} focus on the optimization of definable functions (i.e.,  functions that are definable in an $o$-minimal structure \cite{van1996geometric,davis2020stochastic}), which are all Whitney $C^s$-stratifiable functions for any $s \geq 1$ \cite{van1996geometric}. The finite summation and composition of definable functions are also definable, hence various nonsmooth functions can be easily recognized as definable functions. As shown in \cite{wilkie1996model,davis2020stochastic,bolte2021conservative}, the finite composition among semi-algebraic functions, $\exp$ and $\log$ is definable. Therefore, most common activation functions and loss functions, including sigmoid, hyperbolic tangent, softplus, ReLU \cite{agarap2018deep}, Leaky-ReLU \cite{maas2013rectifier}, piecewise polynomial activations, $\ell_1$-loss, MSE loss, hinge loss, logistic loss and cross-entropy loss are all definable in some o-minimal structures. Furthermore, for any nonsmooth deep neural network built from definable loss functions and activation functions, its objective function is also definable, hence is a potential function for a certain conservative field (e.g., its Clarke subdifferential). 
\end{rmk}

	\begin{rmk}
		Assumption \ref{Assumption_2} implies that $f$ is $M_f$-Lipschitz continuous over $\Rn\times \Rp$, as illustrated in \cite[Remark 3(d)]{bolte2021nonsmooth}. 
	\end{rmk}

Based on Assumption \ref{Assumption_1} and Assumption \ref{Assumption_2}, we make the following definitions on the stationarity of \ref{Prob_Ori} and \ref{Prob_Pen}. 

\begin{defin}[\cite{clarke1990optimization}]
	\label{Defin_FOSP_Ori}
	For any given $(x, y) \in \Rn\times \Rp$, we say $(x, y)$ is a first-order stationary point of \ref{Prob_Ori} if there exists $(d_x, d_y) \in \partial f(x, y)$ such that 
	\begin{equation*}
		\left\{\begin{aligned}
			&0 = d_x - \nabla_{xy}^2 g(x, y) \nabla_{yy}^2 g(x, y)^{-1} d_y , \\
			&0 = \nabla_y g(x, y). 
		\end{aligned}\right.
	\end{equation*}
\end{defin}

Similarly, the stationarity of \ref{Prob_Pen} can be stated in the following definition.
\begin{defin}
	\label{Defin_FOSP_Pen}
	For any given $(x, y) \in \Rn\times \Rp$, we say that $(x, y)$ is a first-order stationary point of \ref{Prob_Pen} if $0 \in \partial h(x, y)$.
\end{defin}

On the other hand, we can characterize the stationarity of \ref{Prob_Ori} from the perspective of conservative field. 

\begin{defin}
	\label{Defin_FOSP_Ori_field}
	Suppose $f$ is a potential function admitted by a convex-valued conservative field $\ca{D}_f$. Then for any given $(x, y) \in \Rn \times \Rp$, we say that $(x, y)$ is a $\ca{D}_f$-stationary point of \ref{Prob_Ori} if there exists $(d_x, d_y) \in \ca{D}_f(x, y)$ such that 
	\begin{equation*}
		\left\{\begin{aligned}
			&0 = d_x - \nabla_{xy}^2 g(x, y) \nabla_{yy}^2 g(x, y)^{-1} d_y, \\
			&0 = \nabla_y g(x, y). 
		\end{aligned}\right.
	\end{equation*}
\end{defin}

It directly follows from Definition \ref{Defin_FOSP_Ori} and Definition \ref{Defin_FOSP_Ori_field} that all the $\partial f$-stationary points of \ref{Prob_Ori} are its first-order stationary points.

\begin{defin}
	\label{Defin_FOSP_Pen_field}
	Suppose $h$ in \ref{Prob_Pen} is a potential function admitted by a convex-valued conservative field $\Dh$,  we say  that $(x, y)\in \Rn \times \Rp$ is a $\Dh$-stationary point of \ref{Prob_Pen} if 
	$0 \in \ca{D}_h(x, y)$.
\end{defin}

	\begin{rmk}
		\label{Rmk_cf_subdifferential}
		$\D_f$ is a generalization of the Clarke subdifferential of $f$, whose expression depends on how to achieve the ``subdifferential'' of $f(x,y)$ \cite{bolte2021conservative,bolte2021nonsmooth}.  As illustrated in \cite{bolte2021nonsmooth}, the conservative field $\D_f$ is not unique and may differs from $\partial f$ in a dense set, hence may lead to infinitely many spurious stationary points for \ref{Prob_Ori}. However, we should keep in mind that the most important case for us is $\ca{D}_f = \partial f$. In the following remark, we discuss how to approximately evaluate $\partial f$ in practice. 
		
		Although in some cases, directly computing one element from $\partial f(x, y)$ may be intractable \cite{bolte2021nonsmooth}, there are already several randomized approaches \cite{burke2002approximating,burke2005a,duchi2012randomized,yousefian2012stochastic,nesterov2017random,burke2020gradient,lin2022gradient} developed for approximately evaluating one element for $\partial f(x, y)$ in practice.  
		
		Some existing approaches approximate $\partial f(x, y)$ by random sampling of gradients \cite{burke2002approximating,burke2005a,burke2020gradient}. In these approaches, with a given radius $\delta$, we randomly sample $\{(x_1, y_1), ...,(x_s, y_s)\} \subset \bb{B}_{\delta}(x, y)$. Since $f$ is differentiable at $\{(x_1, y_1), ...,(x_s, y_s)\}$ almost surely, let $C_{\delta} = \conv \{ \nabla f(x_i, y_i): 1\leq i\leq s \} \subset \partial_{\delta} f(x, y)$, and it holds from Proposition \ref{Prop_Eps_Subdifferential} and \cite[Theorem 3.1]{burke2002approximating}
		that   $ \lim_{\delta \to 0}\mathrm{dist}(\partial f(x, y), C_{\delta}) = 0$. By choosing one element from $C_{\delta}$, we get an approximated evaluation for an element in $\partial f(x, y)$. 
		
		Furthermore, some recent works approximate $\partial f(x, y)$ by the randomized smoothing approaches \cite{duchi2012randomized,yousefian2012stochastic,nesterov2017random,lin2022gradient}. In these approaches, we first uniformly sample $(\zeta_x, \zeta_y) \in \bb{B}_{ {\delta} }(0)$, and approximate $\partial f(x, y)$ by 
		\begin{equation*}
			\tilde{\partial}_{\delta} f(x, y; \zeta_x, \zeta_y) := \frac{n+p}{2 \delta} (f(x + \delta \zeta_x, y + \delta\zeta_y) - f(x - \delta \zeta_x, y - \delta \zeta_y)) \cdot \left[\begin{smallmatrix}
				&\zeta_x \\
				&\zeta_y \\
			\end{smallmatrix}\right]. 
		\end{equation*}
		From \cite[Theorem 3.1, Lemma D.1]{lin2022gradient}, it holds that 
		\begin{equation*}
			\bb{E}_{\zeta_x, \zeta_y} [\tilde{\partial}_{\delta} f(x, y; \zeta_x, \zeta_y)] \in  \partial_{\delta} (x, y), ~ \text{ and }~ \bb{E}_{\zeta_x, \zeta_y} [\norm{\tilde{\partial}_{\delta} f(x, y; \zeta_x, \zeta_y)}^2] \leq 16\sqrt{2\pi}(n+p)M_f^2. 
		\end{equation*}
		Then Proposition \ref{Prop_Eps_Subdifferential} illustrates that $\lim\limits_{\delta \to 0}\mathrm{dist}\left(\partial f(x, y), \bb{E}_{\zeta_x, \zeta_y} [\tilde{\partial}_{\delta} f(x, y; \zeta_x, \zeta_y)] \right) = 0$. Therefore, $\tilde{\partial}_{\delta} f(x, y; \zeta_x, \zeta_y)$ approximates $\partial f(x, y)$ with noises. 
		
	\end{rmk}

\section{Theoretical properties}

\subsection{Equivalence: Clarke subdifferential}
\label{Subsection_equivalence}
In this subsection, we study the equivalence between \ref{Prob_Ori} and \ref{Prob_Pen} based on the Clarke subdifferential. 
We first define 
\begin{align}
	\JAx(x, y) &:= -\nabla_{xy}^2 g(x, y) \nabla_{yy}^2g(x, y)^{-1} + \nabla_{xyy}^3 g(x, y)[\nabla_{yy}^2g(x, y)^{-1}\nabla_y g(x, y)] \nabla_{yy}^2g(x, y)^{-1}, \label{Eq_JAX}\\
	\JAy(x, y) &:=  \nabla_{yyy}^3 g(x, y)[\nabla_{yy}^2g(x, y)^{-1}\nabla_y g(x, y)] \nabla_{yy}^2g(x, y)^{-1} \label{Eq_JAY}.
\end{align}
Then the following proposition characterizes the expression of $\partial h(x, y)$ for any $(x, y) \in \Rn\times \Rp$. 

\begin{prop}
	\label{Prop_gradient_expression}
	For any $x \in \Rn$ and $y \in \Rp$, it holds that 
	\begin{equation*}
		\partial h(x,y) \subseteq \left\{   \left[\begin{smallmatrix}
			&d_x + \JAx(x,y) d_y + \beta \nabla_{xy}^2 g(x, y) \nabla_y g(x, y)\\
			&\JAy(x,y)d_y + \beta \nabla_{yy}^2 g(x, y) \nabla_y g(x, y) \\
		\end{smallmatrix}\right] :   \left[\begin{smallmatrix}
			&d_x \\
			&d_y \\
		\end{smallmatrix}\right] \in \partial f(x, \A(x, y)) \right\}.
	\end{equation*}
	Here the equality holds when $f$ is Clarke regular. 
\end{prop}
Proposition \ref{Prop_gradient_expression} can be verified through direct calculation, hence we omit its proof for simplicity.

\begin{prop}
	\label{Prop_stationary_feasible}
	For any $(x, y) \in \Rn\times \Rp$, suppose $(x, y) \in \M$ is a first-order stationary point of \ref{Prob_Pen}, then $(x,y)$ is a first-order stationary point of \ref{Prob_Ori}.
	
	Furthermore, when $f$ is Clarke regular, then for any given $(x, y) \in \M$, $(x,y)$ is a first-order stationary point of \ref{Prob_Ori} if and only if it is a first-order stationary point of \ref{Prob_Pen}. 
\end{prop}
\begin{proof}
	Since $(x,y) \in \M$ is a first-order stationary point of \ref{Prob_Pen},  it follows from the optimality conditions of \ref{Prob_Pen}  that $0 \in \partial h(x, y)$. Together with the fact that $0 = \nabla_{y} g(x, y)$ and Proposition \ref{Prop_gradient_expression},  there exists $(d_x, d_y) \in \partial f(x, y)$ such that $
	0 = d_x - \nabla_{xy}^2 g(x, y) \nabla_{yy}^2 g(x, y)^{-1} d_y$, 
	which coincides with the optimality conditions of \ref{Prob_Ori}. Therefore, we obtain that $(x,y)$ is a first-order stationary point of \ref{Prob_Ori}.

	Furthermore, when $f$ is assumed to be Clarke regular, and  $(x, y) \in \M$ is a first-order stationary point of \ref{Prob_Ori}, Proposition \ref{Prop_gradient_expression} illustrates that there exists $(d_x, d_y) \in \partial f(x, y)$ such that $
	0 = d_x - \nabla_{xy}^2 g(x, y) \nabla_{yy}^2 g(x, y)^{-1} d_y \in \partial h(x, y)$. 
	Therefore, $(x, y)$ is a first-order stationary point of \ref{Prob_Pen}. 
	This completes the proof. 
\end{proof}

Proposition \ref{Prop_stationary_feasible} illustrates that any first-order stationary point of \ref{Prob_Pen} on $\M$ is also a first-order stationary point of \ref{Prob_Ori}. In the rest of this subsection, we aim to show that with a sufficiently large penalty parameter $\beta$, any first-order stationary point of \ref{Prob_Pen} lies on $\M$. 

\begin{lem}
	\label{Le_lip_constant}
	The Lipschitz constant for $\nabla_{yy}^2 g(x, y)^{-1}$ is no greater than $\frac{Q_g}{\mu^2}$. 
\end{lem}
\begin{proof}
	Firstly, notice that $\norm{(A+t E)^{-1} - \left(A^{-1} - t A^{-1}EA^{-1}\right)} = \ca{O}(t^2)$ holds for any symmetric nonsingular matrix $A$ and any square symmetric matrix $E$.  Therefore, the following inequality holds for any $d_x \in \Rn$
	\begin{equation*}
		\begin{aligned}
			&\norm{\nabla_{yy}^2 g(x+ t d_x, y)^{-1}  - \nabla_{yy}^2 g(x, y)^{-1}}\\
			={}  &t \norm{\nabla_{yy}^2 g(x, y)^{-1} \nabla_{yyx}^3 g(x, y) [d_x] \nabla_{yy}^2 g(x, y)^{-1}} + \ca{O}(t^2) \leq \frac{Q_g}{\mu^2} t\norm{d_x} + \ca{O}(t^2).
		\end{aligned}
	\end{equation*}
	Similarly, for any $d_y \in \Rp$, it holds that 
	\begin{equation*}
		\begin{aligned}
			&\norm{\nabla_{yy}^2 g(x, y+ t d_y)^{-1}  - \nabla_{yy}^2 g(x, y)^{-1}}\\
			={}  &t \norm{\nabla_{yy}^2 g(x, y)^{-1} \nabla_{yyy}^3 g(x, y) [d_y] \nabla_{yy}^2 g(x, y)^{-1}} + \ca{O}(t^2) \leq  \frac{Q_g}{\mu^2} t\norm{d_y} + \ca{O}(t^2).
		\end{aligned}
	\end{equation*}
	Therefore, we can conclude that the Lipschitz constant for $\nabla_{yy}^2 g(x, y)^{-1}$ is no greater than $\frac{Q_g}{\mu^2}$. 
\end{proof}

\begin{lem}
	\label{Le_secondorder_descrease}
	For any given $(x, y) \in \Rn\times \Rp$, it holds that 
	\begin{equation*}
		\norm{\ysol(x) - \A(x, y)} \leq \frac{Q_g}{2\mu^3} \norm{\nabla_y g(x, y)}^2. 
	\end{equation*}
\end{lem}
\begin{proof}
	For any $v \in \bb{R}^p$, it follows from the mean-value theorem that there exists $\xi \in \Rp$ such that 
	\begin{equation*}
		\begin{aligned}
			&v\tp \nabla_y g(x, \A(x, y)) = v\tp \nabla_y g(x, y - \nabla_{yy}^2g(x, y)^{-1}  \nabla_y g(x, y))\\
			={}&  v\tp \nabla_y g(x, y) - v\tp \nabla_{yy}^2g(x, y) \nabla_{yy}^2g(x, y)^{-1}  \nabla_y g(x, y) \\
			&+ \frac{1}{2}v\tp \nabla_{yyy}^3g(x, \xi)[\nabla_{yy}^2g(x, y)^{-1}  \nabla_y g(x, y)  ]\nabla_{yy}^2g(x, y)^{-1}  \nabla_y g(x, y) \\
			\leq{}& \frac{Q_g}{2} \norm{v}\norm{\nabla_{yy}^2g(x, y)^{-1}  \nabla_y g(x, y) }^2
			\leq \frac{Q_g}{2\mu^2}\norm{v} \norm{\nabla_y g(x, y)}^2. 
		\end{aligned}
	\end{equation*}
	
	As a result,  it holds that $ \norm{\nabla_y g(x, \A(x, y))} \leq \frac{Q_g}{2\mu^2}\norm{\nabla_y g(x, y)}^2$.
	Then from the fact that $g(x, y)$ is $\mu$-strongly convex with respect to $y$, we obtain that 
	\begin{equation*}
		\norm{\A(x, y) - \ysol(x)} 
		\leq \frac{1}{\mu}\norm{\nabla_y g(x, \A(x, y))-\nabla_y g(x, \ysol(x))}
		\leq \frac{Q_g}{2\mu^3}\norm{\nabla_y g(x, y))}^2.
	\end{equation*}
	This completes the proof. 
\end{proof}

\begin{rmk}
		Lemma \ref{Le_secondorder_descrease} illustrates that for any $(x, y) \in \Rn \times \Rp$, it holds that $\norm{\nabla_{y}g(x, \A(x, y))} = \ca{O}(\norm{\nabla_y g(x, y)}^2)$. As a result, the mapping $(x, y) \mapsto (x, \A(x, y))$ satisfies the Assumption \ref{Assumption_constraint_dissolving}, and hence it is a constraint dissolving mapping for \ref{Prob_Ori}. 
\end{rmk}

\begin{prop}
	Suppose $\beta \geq \frac{M_f Q_g}{\mu^3}$ and $\Phi(x)$ is bounded below in $\bb{R}^n$. Then $h(x, y)$ is bounded below.  
\end{prop}
\begin{proof}
	We conclude from Lemma \ref{Le_secondorder_descrease} that  
	\begin{equation*}
		\begin{aligned}
			&h(x, y) - f(x, \ysol(x)) = f(x, \A(x, y)) + \frac{\beta}{2} \norm{\nabla_y g(x, y)}^2 - f(x, \ysol(x)) \\
			\geq{}& -M_f \norm{\A(x, y) - \ysol(x)} + \frac{\beta}{2} \norm{\nabla_y g(x, y)}^2 \geq   0,
		\end{aligned}
	\end{equation*}
	which implies that
	\begin{equation*}
		\inf_{(x, y) \in \Rn\times \Rp} h(x, y) \geq \inf_{x \in \M} f(x, \ysol(x)) = \inf_{x \in \Rn} \Phi(x) > -\infty,
	\end{equation*}
	hence completes the proof. 
\end{proof}

\begin{lem}
	\label{Le_Lipschitz_continuous_fA}
	For any given $(x, y) \in \Rn \times \Rp$, and any  $d \in \bb{R}^p$, it holds that 
	\begin{equation*}
		\mathop{ \lim\sup }_{t \to 0} \left|\frac{ f(x, \A(x, y+td)) - f(x, \A(x, y)) }{t}\right| \leq \frac{M_fQ_g}{\mu^2} \norm{\nabla_y g(x, y)}\norm{d}. 
	\end{equation*}
\end{lem}
\begin{proof}
	Let $z_t := y+td$. Then it follows from the expression of $\A$ that 
	\begin{equation*}
		\begin{aligned}
			&\left|f(x, \A(x, z_t)) - f(x, \A(x, y))\right|\\
			={}& \left|f(x, z_t - \nabla_{yy}^2 g(x, z_t)^{-1} \nabla_y g(x, z_t) ) - f(x, y - \nabla_{yy}^2g(x, y)^{-1} \nabla_y g(x, y))\right|\\
			\leq{}& \left|f(x, z_t - \nabla_{yy}^2 g(x, y)^{-1} \nabla_y g(x, z_t) ) - f(x, y - \nabla_{yy}^2 g(x, y)^{-1} \nabla_y g(x, y))\right| \\
			+& \left|f(x, z_t - \nabla_{yy}^2 g(x, z_t)^{-1} \nabla_y g(x, z_t) )  -  f(x, z_t - \nabla_{yy}^2 g(x, y)^{-1} \nabla_y g(x, z_t))\right|.
		\end{aligned}
	\end{equation*}
	Notice that 
	\begin{equation*}
		\norm{td - \nabla_{yy}^2 g(x, y)^{-1} \nabla_y g(x, z_t) + \nabla_{yy}^2g(x, y)^{-1} \nabla_y g(x, y)} \leq t^2\frac{Q_g}{\mu} \norm{d}^2,
	\end{equation*}
	hence we achieve the following inequality, 
	\begin{equation*}
		\begin{aligned}
			&\left|f(x, z_t - \nabla_{yy}^2 g(x, y)^{-1} \nabla_y g(x, z_t) ) - f(x, y - \nabla_{yy}^2g(x, y)^{-1} \nabla_y g(x, y))\right| \\
			\leq{}& M_f \norm{\left(z_t - \nabla_{yy}^2 g(x, y)^{-1} \nabla_y g(x, z_t)\right) - \left(y - \nabla_{yy}^2g(x, y)^{-1} \nabla_y g(x, y)\right)  }\\
			\leq{}& t^2\frac{M_fQ_g}{\mu} \norm{d}^2.
		\end{aligned}
	\end{equation*}
	On the other hand, 
	\begin{equation*}
		\begin{aligned}
			&\left|f(x, z_t - \nabla_{yy}^2 g(x, z_t)^{-1} \nabla_y g(x, z_t) ) - f(x, z_t - \nabla_{yy}^2g(x, y)^{-1} \nabla_y g(x, z_t))\right|\\
			&\leq M_f \norm{\nabla_{yy}^2 g(x, z_t)^{-1} \nabla_y g(x, z_t) - \nabla_{yy}^2 g(x, y)^{-1} \nabla_y g(x, z_t)} \\
			&\leq M_f \norm{\nabla_{yy}^2 g(x, z_t)^{-1} - \nabla_{yy}^2 g(x, y)^{-1}} \norm{\nabla_y g(x, z_t)}\\
			&\leq t\frac{Q_gM_f}{\mu^2}  \norm{\nabla_y g(x, z_t)}\norm{d}
			\leq  t\frac{Q_gM_f}{\mu^2}  \norm{\nabla_y g(x, y)}\norm{d} + t^2 \frac{Q_gM_f L_g}{\mu^2} \norm{d}^2. 
		\end{aligned}
	\end{equation*}
	Therefore, we obtain that
	\begin{equation*}
		\mathop{ \lim\sup }_{t \to 0} \left|\frac{ f(x, \A(x, y+td)) - f(x, \A(x, y)) }{t}\right|
		\leq  \frac{Q_gM_f}{\mu^2}  \norm{\nabla_y g(x, y)}\norm{d},
	\end{equation*}
	and the proof is completed. 
\end{proof}

\begin{theo}
	\label{Theo_subgradient_equivalence}
	Suppose $\beta \geq \frac{2Q_gM_f}{\mu^3}$. If $(x,y)\in \bb{R}^n\times \bb{R}^p$ is a first-order stationary point of \ref{Prob_Pen}, then $(x,y) \in \M$ and hence is a first-order stationary point of \ref{Prob_Ori}. 
\end{theo}
\begin{proof}
	Suppose $(x, y)$ is a stationary point of \ref{Prob_Pen}, we have $0 \in \partial h(x, y)$. Therefore,  it follows from Definition \ref{Defin_Subdifferential} that $0 \leq  h^\circ(x, y; 0, -\nabla_y g(x, y))$.
	
	Notice that $\nabla_y g(x, y)$ is differentiable, then it holds that 
	\begin{equation*}
		\begin{aligned}
			&\mathop{\lim}_{t\to 0} \frac{\norm{\nabla_y g(x, y- t\nabla_y g(x, y))   }^2 - \norm{\nabla_y g(x, y)   }^2}{t} \\
			={}& - 2 \nabla_y g(x, y)\tp \nabla_{yy}^2 g(x, y) \nabla_y g(x, y)  \leq -2\mu \norm{\nabla_y g(x, y)}^2. 
		\end{aligned}
	\end{equation*}
	Therefore, it holds from Lemma \ref{Le_Lipschitz_continuous_fA}  that 
	\begin{equation*}
		\small
		\begin{aligned}
			&0 \leq h^\circ(x, y; 0, -\nabla_y g(x, y)) 
			= \mathop{\lim\sup}_{(\tilde{x}, \tilde{y}) \to (x,y), ~t \downarrow 0}~ \frac{ h(\tilde{x}, \tilde{y}-t\nabla_y g(x, y)) - h(\tilde{x}, \tilde{y}) }{t}\\
			={}& \mathop{\lim\sup}_{(\tilde{x}, \tilde{y}) \to (x,y), ~t \downarrow 0}~ \frac{ f(\tilde{x}, \A(\tilde{x},\tilde{y}-t\nabla_y g(x, y))) - f(\tilde{x}, \A(\tilde{x}, \tilde{y})) }{t}\\
			& + \frac{\beta}{2}\lim_{t\to 0} \frac{\norm{\nabla_y g(x, y- t\nabla_y g(x, y))   }^2 - \norm{\nabla_y g(x, y)   }^2}{t} \\
			\leq{}& -\mu\beta \norm{\nabla_y g(x, y)}^2 +   \mathop{\lim\sup}_{(\tilde{x}, \tilde{y}) \to (x,y)}~ \frac{M_fQ_g}{\mu^2}   \norm{\nabla_y g(\tilde{x}, \tilde{y})}\norm{\nabla_y g(x, y)}\\
			\leq{}& -\frac{\mu\beta}{2} \norm{\nabla_y g(x, y)}^2 \leq   0.
		\end{aligned}
	\end{equation*}
	Therefore, we conclude that  $\nabla_y g(x,y) = 0$ and $(x, y) \in \M$. Thus $(x, y)$ is a first-order stationary point of \ref{Prob_Ori} by Proposition \ref{Prop_stationary_feasible}. 
\end{proof}


\begin{coro}
	Suppose $f$ is Clarke regular and $\beta \geq \frac{2Q_gM_f}{\mu^3}$. Then \ref{Prob_Ori} and \ref{Prob_Pen} have the same first-order stationary points over $\Rn\times \Rp$. 
\end{coro}
The proof  straightforwardly follows from Theorem \ref{Theo_subgradient_equivalence} and Proposition \ref{Prop_stationary_feasible}. Hence we omit its details for simplicity. 

\subsection{Equivalence: conservative field}
In this subsection, we study the equivalence between \ref{Prob_Ori} and \ref{Prob_Pen} based on the concept of  conservative field. 
With the set-valued mapping $\Dh(x, y)$ defined by 
\begin{equation}
	\label{Eq_Defin_Dh}
	\ca{D}_h(x, y) := \left\{   \left[\begin{smallmatrix}
		&d_x + \JAx(x,y) d_y + \beta \nabla_{xy}^2 g(x, y) \nabla_y g(x, y)\\
		&\JAy(x,y)d_y + \beta \nabla_{yy}^2 g(x, y) \nabla_y g(x, y) \\
	\end{smallmatrix}\right] :   \left[\begin{smallmatrix}
		&d_x \\
		&d_y \\
	\end{smallmatrix}\right] \in \ca{D}_f(x, \A(x, y)) \right\},
\end{equation}
we have the following proposition characterizing the property of $\ca{D}_h$.
\begin{prop}
	\label{potential}
	$\ca{D}_h(x, y)$ is a convex-valued conservative field that admits $h(x,y)$ as its potential. 
\end{prop}
\begin{proof}
	Since $\A$ is continuously differentiable, it holds that  $\A$ is a potential mapping for its Jacobian $[J_{A,x}(x, y), J_{A, y}(x, y)]\tp$.  As a result, by the chain rule and sum rule in Proposition \ref{Prop_Conservative_Chain_rule} and Proposition \ref{Prop_Conservative_Sum_rule}, $\Dh$  
	is a conservative field that admits $h(x, y)$ as its potential function. 
\end{proof}


\begin{prop}
	\label{Prop_conservative_feasible}
	For any given $(x, y) \in \M$, $(x, y)$ is a $\ca{D}_f$-stationary point of \ref{Prob_Ori} if and only if $(x, y)$ is a $\Dh$-stationary point of \ref{Prob_Pen}. 
\end{prop}
From Definition \ref{Defin_FOSP_Pen_field}, any $(x,y)$ satisfying $0 \in \Dh(x,y)$ is called a $\Dh$-stationary point of \ref{Prob_Pen}.  Then Proposition \ref{Prop_conservative_feasible} directly follows from the expression of $\Dh$, and we omit its proof for simplicity.

\begin{theo}
	\label{Theo_conservative_equivalence}
	Suppose $\beta \geq \frac{2Q_gM_f}{\mu^3}$, then $(x, y) \in \Rn\times \Rp$ is a $\ca{D}_f$-stationary point of \ref{Prob_Ori} if and only if $(x, y)$ is a $\Dh$-stationary point of \ref{Prob_Pen}. 
\end{theo}
\begin{proof}
	For any $(x, y) \in \Rn\times \Rp$ and any $(d_x, d_y) \in \ca{D}_f(x, \A(x, y))$,   the inclusion $0 \in \Dh(x, y)$ implies that there exists $(d_x, d_y) \in \ca{D}_f(x, \A(x, y))$ such that 
	\begin{equation*}
		\begin{aligned}
			&0 = d_x + \JAx(x,y) d_y + \beta \nabla_{xy}^2 g(x, y) \nabla_y g(x, y),\\
			&0 = \JAy(x, y) d_y + \beta \nabla_{yy}^2 g(x, y) \nabla_y g(x, y). 
		\end{aligned}
	\end{equation*}
	From \eqref{Eq_JAY} and Assumption \ref{Assumption_2}, it holds that $\norm{\JAy(x, y) d_y} \leq \frac{Q_gM_f}{\mu^2}\norm{\nabla_y g(x, y)}$. 
	Then we obtain that
	\begin{equation*}
		\begin{aligned}
			0 ={}& \norm{\JAy(x, y) d_y + \beta \nabla_{yy}^2 g(x, y) \nabla_y g(x, y)}\\
			\geq{}& \norm{\beta \nabla_{yy}^2 g(x, y) \nabla_y g(x, y)} - \norm{\JAy(x, y) d_y}\\
			\geq{}& \left(\mu\beta -\frac{Q_gM_f}{\mu^2}\right) \norm{\nabla_y g(x, y)}
			\geq \frac{\mu\beta}{2} \norm{\nabla_y g(x, y)},
		\end{aligned}
	\end{equation*}
	which shows that  $\nabla_y g(x, y) = 0$. Hence $(x, y) \in \M$. 
	Therefore, Proposition \ref{Prop_conservative_feasible} illustrates that $(x, y)$ is a $\ca{D}_f$-stationary point of \ref{Prob_Ori}. 
	
	On the other hand, when $(x, y)$ is a $\ca{D}_f$-stationary point of \ref{Prob_Ori},  Proposition \ref{Prop_conservative_feasible} shows that $0 \in \Dh(x, y)$. Hence  $(x, y)$ is a $\Dh$-stationary point of \ref{Prob_Pen} by Definition \ref{Defin_FOSP_Pen_field}. 
\end{proof}

	As illustrated in Remark \ref{Rmk_cf_subdifferential}, the most important example for us is $\ca{D}_f = \partial f$, and it is usually easy to compute an approximation for $\partial f(x, y)$ in practice through some randomized approaches \cite{burke2002approximating,burke2005a,duchi2012randomized,yousefian2012stochastic,nesterov2017random,burke2020gradient,lin2022gradient}. Therefore, we present the following corollary to illustrate the equivalence between \ref{Prob_Ori} and \ref{Prob_Pen} when we choose $\D_f$ as $\partial f$ in \eqref{Eq_Defin_Dh}. 
	\begin{coro}
		\label{Coro_conservative_equivalence}
		Suppose $\beta \geq \frac{2Q_gM_f}{\mu^3}$ and $\D_h$ is chosen by \eqref{Eq_Defin_Dh} with $\D_f = \partial f$. Then $(x, y) \in \Rn\times \Rp$ is a first-order stationary point of \ref{Prob_Ori} if and only if $(x, y)$ is a $\Dh$-stationary point of \ref{Prob_Pen}. 
	\end{coro}
	\begin{proof}
		When $\D_f = \partial f$ in \eqref{Eq_Defin_Dh}, the corresponding conservative field $\D_h$ is set as 
		\begin{equation*}
			\D_h(x, y) = \left\{   \left[\begin{smallmatrix}
				&d_x + \JAx(x,y) d_y + \beta \nabla_{xy}^2 g(x, y) \nabla_y g(x, y)\\
				&\JAy(x,y)d_y + \beta \nabla_{yy}^2 g(x, y) \nabla_y g(x, y) \\
			\end{smallmatrix}\right] :   \left[\begin{smallmatrix}
				&d_x \\
				&d_y \\
			\end{smallmatrix}\right] \in \partial_f(x, \A(x, y)) \right\}. 
		\end{equation*} 
		Then from Theorem \ref{Theo_conservative_equivalence} and Definition \ref{Defin_FOSP_Ori}, whenever $(x, y)$ is a $\Dh$-stationary point of \ref{Prob_Pen} with $\D_f = \partial f$ in \eqref{Eq_Defin_Dh}, $(x, y)$ is a $\partial f$-stationary point, and hence it is a first-order stationary point of \ref{Prob_Ori}.  
		
		On the other hand, when $(x, y)$ is a first-order stationary point of \ref{Prob_Ori},  Proposition \ref{Prop_conservative_feasible} directly shows that $0 \in \D_h(x, y)$, thus we complete the proof. 
	\end{proof}

\section{Algorithmic Design}
Subgradient method and its variants play important roles in minimizing nonsmooth functions that are not necessarily regular, particularly in training deep neural networks involving nonsmooth activation functions. Recently, \cite{davis2020stochastic} shows the global convergence for applying subgradient methods in minimizing nonsmooth functions based on their Clarke subdifferentials. Moreover, \cite{bolte2021conservative} introduces the concept of conservative field, which overcomes the limitations of Clarke subdifferential, and further explains the behavior of stochastic subgradient methods  when they are applied to train nonsmooth neural networks with automatic differentiation algorithms.  Furthermore, \cite{bolte2021conservative,castera2021inertial,bianchi2022convergence} establish the convergence properties for some subgradient methods that are developed from the conservative field of the objective function, as they are implemented in practice. 

In this section, we aim to design subgradient methods to solve \ref{Prob_Pen} based on the formulation of $\Dh$. 
In Proposition \ref{potential}, we show that $\Dh$ is a conservative field that admits $h$ as the potential function. Then various existing subgradient approaches \cite{bolte2021conservative,castera2021inertial,bianchi2022convergence} can be directly applied to \ref{Prob_Pen} from the explicit formulation of $\Dh$. However, it may be expensive to calculate the $\nabla_{xyy}^3g$ and $\nabla_{yyy}^3g$ in practice, hence computing $\Dh(x, y)$ exactly may be expensive and impractical. 

To this end, we first propose a general framework for applying subgradient methods to solve \ref{Prob_Pen}, which enables the inexact evaluation of $\Dh$. Then we propose several different set-valued mappings $\hDh$, $\hDp$ and $\hDs$, all of which approximates $\Dh$ and avoid computing the third-order derivatives of $g$. Based on these set-valued mappings, we design several subgradient methods that adopt inexact evaluations to achieve better efficiency.  Moreover, we demonstrate that the global convergence for these subgradient-based methods directly follows from the proposed framework in Section \ref{Subsection_framework}.

\subsection{A unified framework for subgradient-based methods}
\label{Subsection_framework}
%

In this subsection, we utilize the conservative field $\Dh$ to develop a framework for applying subgradient methods to solve \ref{Prob_Pen}. 
We first consider the iteration sequence $\{(\xk, \yk)\}$ generated by the following updating scheme that generalizes the subgradient methods,
\begin{equation}
	\label{Iter_framework}
	\xkp = \xk - \etak\left( u_{x,k} + \xi_{x,k} \right),\quad \text{and} \quad 
	\ykp = \yk - \etak\left( u_{y,k} + \xi_{y,k} \right).
\end{equation}
Here $\etak>0$ refers to the stepsize, $(u_{x,k}, u_{y,k})$ should be thought as an approximate descent direction for $h(x, y)$ at $(x_k, y_k)$, Moreover, $\xi_{x,k}$ and $\xi_{y,k}$ denote the ``errors'' introduced by stochasticity and inexact evaluation. Similar to \cite{davis2020stochastic}, we stipulate the following assumptions on \eqref{Iter_framework}.
\begin{assumpt}
	\label{Assumption_framework}
	
	\begin{enumerate}[label=(\alph*)]
		\item The generated iterates $\{(\xk, \yk)\}$ are uniformly bounded: 
		$\sup_{k >0} \norm{\xk} + \norm{\yk} <+\infty$.
		\item The stepsizes are nonnegative, square summable, but not summable:
		\begin{equation*}
			\etak >0,\quad \sum_{k = 0}^{+\infty} \etak = +\infty, \quad \text{and} \quad \sum_{k = 0}^{+\infty} \etak^2 < +\infty. 
		\end{equation*}
		\item The series of weighted noise is convergent. That is, there exists $v_{x} \in \bb{R}^n$ and $v_y \in \bb{R}^p$, such that $\lim\limits_{N\to +\infty} ~ \sum\limits_{k = 0}^{N} \etak\xi_{x,k} = v_x$ and $\lim\limits_{N\to +\infty} ~ \sum\limits_{k = 0}^{N} \etak\xi_{y,k} = v_y$.
		\item There exists a set-valued mapping $\ca{D}(x, y)$ that has closed graph and compact convex values.  Moreover, $\ca{D}$ has the property that for any sequence $\{(x_{k_j}, y_{k_j})\}$ that converges to a point $(\tilde{x}, \tilde{y})$ and  any unbounded increasing sequence $\{k_j\}$, it holds that $
		\lim\limits_{N\to +\infty}\mathrm{dist}\left( \frac{1}{N}\sum\limits_{j = 1}^{N}\left[\begin{matrix}
			u_{x, k_j} \\
			u_{y, k_j}
		\end{matrix}\right] , \ca{D}(\tilde{x}, \tilde{y})\right) = 0$.
		\item The set $\{ h(x, y): 0 \in \ca{D}(x, y) \}$ has empty interior, i.e. its complementary is dense in $\bb{R}$. 
		\item There exists a constant $\delta>0$ such that for any $(x, y) \in \Rn \times \Rp$ and  any $w \in \ca{D}(x, y)$, it holds that $
		\sup_{\zeta \in \Dh(x, y)}~ \zeta\tp w \geq \delta \norm{w}^2$. 
	\end{enumerate}
\end{assumpt}

Assumption \ref{Assumption_framework}(a)-(b) are common assumptions in various  existing works \cite{davis2020stochastic,bolte2021conservative, castera2021inertial}. Assumption \ref{Assumption_framework}(c) is a mild assumption that controls the growth of the noise sequence $\{(\xi_{x,k}, \xi_{y,k})\}$ as the stepsize  decreases, which can be satisfied by the stochastic subgradient method described in \cite{davis2020stochastic}.  Moreover, Assumption \ref{Assumption_framework}(d) illustrates how $(u_{x,k}, u_{y,k})$ approximates $\ca{D}(\xk, \yk)$.  Assumption \ref{Assumption_framework}(e) is the weak Sard's condition \cite[Assumption B(1), Assumption F(1)]{davis2020stochastic}, which holds whenever $h$ is definable and $\D = \partial h$ \cite[Lemma 5.7]{davis2020stochastic}.   Furthermore,  Assumption \ref{Assumption_framework}(f) implies the descent condition in \cite[Assumption B(2)]{davis2020stochastic}, as illustrated in the following proposition.

\begin{prop}
	\label{Prop_path_descrease_framework}
	Suppose Assumption \ref{Assumption_framework} holds. Let $\gamma: \bb{R}_+ \to \Rn \times \Rp$ be any absolutely continuous path such that  the differential inclusion $\gamma'(t) \in - \ca{D}(\gamma(t))$ holds for  a.e. $t \in \bb{R}_+$. 
	Then the following inequality holds for any $t >0$,
	\begin{equation*}
		h(\gamma(t))- h(\gamma(0)) \leq -\delta\int_{0}^t \mathrm{dist}\left( 0, \ca{D}(\gamma(\tau)) \right)^2  \mathrm{d}\tau. 
	\end{equation*}
\end{prop}
\begin{proof}
	Notice that $h$ is the potential function of the conservative field $\Dh$. Therefore, it follows from Definition \ref{Defin_conservative_field_path_int} that 
	\begin{equation*}
		\begin{aligned}
			h(\gamma(t))-  h(\gamma(0)) =  \int_{0}^t \inf_{\zeta \in \Dh(\gamma(t))}\inner{\zeta, -\gamma'(t)} \mathrm{d}\tau
			\leq -\delta\int_{0}^t  \mathrm{dist}\left( 0, \ca{D}(\gamma(\tau)) \right)^2  \mathrm{d}\tau,
		\end{aligned}
	\end{equation*}
	and this completes the proof. 
\end{proof}

	\begin{prop}
		\label{Prop_closed_graph_convex_approx}
		For any set-valued mapping $\ca{D}: \bb{R}^m \rightrightarrows \bb{R}^s$, suppose $\ca{D}$ is compact and convex valued and has closed graph, then for any $\tilde{w} \in \bb{R}^m$, any sequence $\{w_k\}$ that converges to $\tilde{w}$ and any $\{u_k\}$ that satisfies $ \lim\limits_{k\to +\infty}  \mathrm{dist}\left(u_k, \ca{D}(w_k) \right) = 0$, it holds that 
		\begin{equation*}
			\lim_{N\to +\infty}\mathrm{dist}\left( \frac{1}{N}\sum_{k = 1}^{N}  u_k , \ca{D}(\tilde{w})\right) = 0.
		\end{equation*}
	\end{prop}
	\begin{proof}
		We first assume that 
		the argument to be proved is not true. Then there exists a constant $\varepsilon_0 >0$,
		a sequence  $\{w_k\}$ converging to $\tilde{w}$, a sequence $\{u_k\}$ that satisfies $ \lim\limits_{k\to +\infty}  \mathrm{dist}\left(u_k, \ca{D}(w_k) \right) = 0$ and a sequence $\{N_j\} \subset \ca{N}$ satisfying $N_j \to +\infty$, such that
		\begin{equation}
			\label{Eq_temp_1}
			\mathrm{dist}\left(\frac{1}{N_j}\sum_{k = 1}^{N_j} u_k , \ca{D}(\tilde{w})\right) \geq \varepsilon_0. 
		\end{equation}
		
		From the convexity of $\ca{D}(\tilde{w})$, we conclude that for any $j \geq 0$, there exists an index $k_j \leq N_j$ such that 
		\begin{equation}
			\label{Eq_temp_2}
			\mathrm{dist}\left(u_{k_j} , \ca{D}(\tilde{w})\right) \geq \frac{\varepsilon_0}{2}. 
		\end{equation}
		We claim that we can always choose a sequence $\{k_j\}$ such that $k_j \to +\infty$. Otherwise, for any  
		$
		N > \sup_{j \geq 0} k_j + \left(\frac{2}{\varepsilon_0} \right) \sum_{i = 1}^{\sup_{j \geq 0} k_j} \mathrm{dist}\left( u_i, \ca{D}(\tilde{w}) \right) ,
		$
		it holds that 
		\begin{equation*}
			\small
			\begin{aligned}
				&\mathrm{dist}\left( \frac{1}{N}\sum_{k = 1}^{N} u_k , \ca{D}(\tilde{w})\right) \leq \frac{1}{N}\sum_{k = 1}^{N} \mathrm{dist} \left( u_k, \ca{D}(\tilde{w}) \right) 
				\leq \frac{1}{N } \sum_{k = 1}^{\sup_{j \geq 0} k_j } \mathrm{dist} \left( u_k, \ca{D}(\tilde{w}) \right) + \frac{\varepsilon_0}{2} < \varepsilon_0,
			\end{aligned}
		\end{equation*}
		which contradicts \eqref{Eq_temp_1} and further verifies our claim. 
		
		Therefore, for the selected sequence of indices $\{k_j\}$, it holds that $\lim\limits_{j\to +\infty} w_{k_j} = \tilde{w}$, and $$\lim_{j\to +\infty} \mathrm{dist}\left((w_{k_j}, u_{k_j}) , \mathrm{graph}(\ca{D})\right) = 0.$$
		Since $\ca{D}$ has closed graph,  any cluster point of $\{u_{k_j}\}$ lies in $\ca{D}(\tilde{w})$, which further leads to $$
		\mathop{\lim\inf}_{j\to +\infty} ~	\mathrm{dist}\left(u_{k_j} , \ca{D}(\tilde{w})\right) = 0.$$ But this contradicts \eqref{Eq_temp_2}. Thus the proof is completed by contradiction. 
	\end{proof}

\begin{theo}
	\label{Theo_framework_convergence}
	Suppose Assumption \ref{Assumption_framework} holds. Then for the sequence $\{(\xk, \yk)\}$ generated from \eqref{Iter_framework}, all its limit point lies in $\{ (x, y) \in \Rn\times \Rp: 0 \in \ca{D}(x, y) \}$. Moreover, the sequence of function values $\{h(\xk, \yk)\}$ converges. 
\end{theo}
\begin{proof}
	Assumption \ref{Assumption_framework}(a)-(d) imply the validity of Assumption A in \cite{davis2020stochastic}. Moreover, Assumption \ref{Assumption_framework}(e)-(f) and Proposition \ref{Prop_path_descrease_framework} show that the Assumption B in \cite{davis2020stochastic} holds. Then  the  proof directly follows from Theorem 3.2 in  \cite{davis2020stochastic}. 
\end{proof}

\subsection{Basic subgradient methods}
In this subsection, we first propose a set-valued mapping $\hDh(x, y)$ that has compact values and  satisfies Assumption \ref{Assumption_framework}(f). Based on $\hDh(x, y)$, we develop a subgradient method as illustrated in Algorithm \ref{Alg:subgradient_method}, where the update direction in each iteration is approximately chosen from $\hDh(x, y)$. Then we establish the global convergence of Algorithm \ref{Alg:subgradient_method} directly from our proposed framework. 

\begin{defin}
	For any given $(x, y) \in \Rn\times \Rp$, we define the set-valued mapping $\hDh: \Rn\times \Rp \rightrightarrows \Rn\times \Rp$ as 
	\begin{equation*}
		\hDh(x, y) := \left\{\left[ \begin{smallmatrix}
			&d_x - \nabla_{xy}^2g(x, y) \left(\nabla_{yy}^2g(x, y)^{-1} d_y - \beta \nabla_y g(x, y)\right) \\
			&\beta  \nabla_y g(x, y)
		\end{smallmatrix} \right]:    \left[\begin{smallmatrix}
			&d_x \\
			&d_y \\
		\end{smallmatrix}\right] \in \ca{D}_f(x, \A(x, y))
		\right\}. 
	\end{equation*}
	
\end{defin}

It is easy to verify that $\hDh$ has closed graph. Moreover, compared with $\Dh$, the formulation of $\hDh$ avoids the third-order derivatives of $g$. Therefore, computing an element from $\hDh$ can be potentially more efficient than directly computing one from $\Dh$. 

\begin{prop}
	\label{Prop_field_approx}
	Suppose $\beta \geq \frac{2M_fQ_g}{\mu^3}$. Then for any given $(x, y) \in \Rn\times \Rp$, it is a $\Dh$-stationary point of \ref{Prob_Pen} if and only if $0 \in \hDh(x, y)$. 
\end{prop}
\begin{proof}
	When $0 \in \hDh(x, y)$, we first conclude that  $  \nabla_y g(x, y) = 0$, which results in the inclusion $(x, y) \in \M$. 
	Moreover,  $0 \in \hDh(x, y)$ implies that there exists $ (d_x, d_y) \in \ca{D}_f(x, y)$
	such that 	$
	d_x - \nabla_{xy}^2g(x, y) \nabla_{yy}^2g(x, y)^{-1} d_y = 0$. 
	Therefore, it follows from Definition \ref{Defin_FOSP_Ori_field} that $(x, y)$ is a $\Dh$-stationary point of \ref{Prob_Pen}. 
	
	On the other hand, when $(x, y)$ is a $\Dh$-stationary point of \ref{Prob_Pen}, from Theorem \ref{Theo_conservative_equivalence}, it holds that $(x,y) \in \M$. Therefore, from the expression of $\Dh(x,y)$ and $ \hDh(x,y)$, we obtain that $0 \in \Dh(x,y) = \hDh(x,y)$ and  the proof is completed.
\end{proof}

\begin{prop}
	\label{Prop_descrease_approx}
	Suppose $\beta \geq \frac{2M_fQ_g}{\mu^3}$. Then for any given $(x, y) \in \Rn\times \Rp$ and $w \in \hDh(x, y)$, it holds that 
	\begin{equation*}
		\sup_{\zeta \in \Dh(x, y)} \inner{ w, \zeta } \geq \min\left\{1, \frac{(2-\sqrt{2})\mu}{2}\right\} \norm{w}^2. 
	\end{equation*}
\end{prop}
\begin{proof}
	For any $(d_x, d_y) \in \ca{D}_f(x, \A(x, y))$, let $$
	w = \left[ \begin{smallmatrix}
		&d_x - \nabla_{xy}^2g(x, y) \left(\nabla_{yy}^2g(x, y)^{-1} d_y - \beta \nabla_y g(x, y)\right) \\
		&\beta  \nabla_y g(x, y)
	\end{smallmatrix} \right] \in  \hDh(x, y), $$
	and define 
	\begin{equation*}
		\small
		z_1 = \left[ \begin{smallmatrix}
			&d_x - \nabla_{xy}^2g(x, y) \left(\nabla_{yy}^2g(x, y)^{-1} d_y - \beta \nabla_y g(x, y)\right) \\
			&\beta \nabla_{yy}^2g(x, y) \nabla_y g(x, y)
		\end{smallmatrix} \right],
		z_2 = 
		\left[ \begin{smallmatrix}
			& \nabla_{xyy}^3 g(x, y)[\nabla_{yy}^2g(x, y)^{-1}\nabla_y g(x, y)] \nabla_{yy}^2g(x, y)^{-1} d_y \\
			&  \nabla_{yyy}^3 g(x, y)[\nabla_{yy}^2g(x, y)^{-1}\nabla_y g(x, y)] \nabla_{yy}^2g(x, y)^{-1} d_y
		\end{smallmatrix} \right].
	\end{equation*}
	Then from the expression of $\Dh$, we have $ z_1 + z_2 \in \Dh(x, y)$.
	Moreover, the expression of $w$ and Lemma \ref{Le_lip_constant}  implies  $\norm{w} \geq \beta \norm{\nabla_y g(x, y)}$ and  $\norm{z_2} \leq \frac{\sqrt{2}Q_g M_f}{\mu^2} \norm{\nabla_y g(x, y)}$.
	As a result, we obtain  
	\begin{equation*}
		\begin{aligned}
			&\inner{w, z_1 + z_2} 
			\geq \norm{ d_x - \nabla_{xy}^2g(x, y) \left(\nabla_{yy}^2g(x, y)^{-1} d_y - \beta \nabla_y g(x, y)\right)}^2 \\
			&+ \beta^2 \mu \norm{\nabla_y g(x, y)}^2  - \frac{\sqrt{2}Q_g M_f}{\mu^2} \norm{\nabla_y g(x, y)}\norm{w}\\
			\geq{}& \norm{ d_x - \nabla_{xy}^2g(x, y) \left(\nabla_{yy}^2g(x, y)^{-1} d_y - \beta \nabla_y g(x, y)\right)}^2 + \frac{(2-\sqrt{2})\beta^2 \mu}{2} \norm{\nabla_y g(x, y)}^2\\
			\geq{}& \min\left\{1, \frac{(2-\sqrt{2})\mu}{2}\right\} \norm{w}^2,
		\end{aligned}
	\end{equation*}
	and this completes the proof. 
\end{proof}

\begin{algorithm}[h]
	\begin{algorithmic}[1]   
		\Require  Function $f$, $g$, initial point $x_0$, $y_0$.
		\For{k = 1,2,...}
		\State Compute $w_k$ by approximately evaluating $\nabla_{yy}^2 g(\xk, \yk)^{-1} \nabla_y g(\xk, \yk)$ such that $ 
		\norm{\nabla_{yy}^2 g(\xk, \yk) w_k - \nabla_y g(\xk, \yk) } \leq \varepsilon_{1,k} $.
		\State Choose $(d_{x,k}, d_{y,k})$ as an approximated evaluation of $\ca{D}_f(\xk, \yk - w_k)$.
		\State Compute $v_k$ such that $
		\norm{\nabla_{yy}^2 g(\xk, \yk) v_k - d_{y,k}} \leq \varepsilon_{2,k}$.
		\State Update $x_k$ and $y_k$ by 
		\begin{align*}
			&\xkp =  \xk - \eta_k \left( d_{x,k} -   \nabla_{xy}^2 g(\xk,\yk) \left( v_k - \beta \nabla_y g(\xk,\yk)\right) \right), \\
			&\ykp =  \yk - \eta_k \beta \nabla_y g(\xk,\yk).
		\end{align*}
		\EndFor
		\State Return $x_k$ and $y_k$.
	\end{algorithmic}  
	\caption{Basic subgradient method for solving \ref{Prob_Pen}.}
	\label{Alg:subgradient_method}
\end{algorithm}

With the definition of $\hDh$, Proposition \ref{Prop_field_approx} and  Proposition
\ref{Prop_descrease_approx}, we can now 
present a basic subgradient method for solving  \ref{Prob_Pen} in Algorithm 1.
We observe that in Algorithm \ref{Alg:subgradient_method}, the search direction
$\left[\begin{array}{c}
	d_{x,k} -  \nabla_{xy}^2 g(\xk,\yk) \left( v_k - \beta \nabla_y g(\xk,\yk)\right)\\
	\nabla_y g(\xk,\yk) \end{array} \right]$ is an element that is approximately in $\hDh(x_k,y_k)$.

To establish the convergence of Algorithm \ref{Alg:subgradient_method}, we need the following assumption.

\begin{assumpt}
	\label{Assumption_Sto_subgrad_basic}
	In Algorithm \ref{Alg:subgradient_method}, we assume
	\begin{enumerate}[label=(\alph*)]
		\item The iterates are uniformly bounded:
		$\sup_{k >0} \norm{\xk} + \norm{\yk} <+\infty$.
		\item The stepsize is nonnegative, square summable, but not summable:
		\begin{equation}
			\label{Eq_Assumption_Sto}
			\etak \geq 0,\quad \sum_{k = 0}^{+\infty} \etak = +\infty, \quad \text{and}\quad \sum_{k = 0}^{+\infty} \etak^2 < +\infty. 
		\end{equation}
		\item The set $\{ f(x, y): \text{$(x, y)$ is a $\ca{D}_f$-stationary point of \ref{Prob_Ori}} \}$ has empty interior.
		\item Let the filtration $\{\ca{F}_k\}$ be the collection of the increasing $\sigma$-fields, i.e., $$\ca{F}_k:= \sigma((x_j, y_j, d_{x,j}, d_{y,j}): j < k).$$ There exists a constant $M_{\sigma}$ such that the approximated evaluation $(d_{x,k}, d_{y,k})$ satisfies the following inequalities,
			\begin{align*}
				&\bb{E}\left[ \norm{ \left(d_{x,k} - \bb{E}[d_{x,k} |\ca{F}_k], ~d_{y,k} - \bb{E}[d_{y,k} |\ca{F}_k]\right)}^2  \Big| \ca{F}_k \right] \leq M_{\sigma}, \quad \text{for any $k \geq 1$},\\
				&\lim_{k\to +\infty} \mathrm{dist}\left( \D_f(x_k, y_k - w_k), \left(\bb{E}[d_{x,k} |\ca{F}_k], ~ \bb{E}[d_{y,k} |\ca{F}_k]\right) \right) = 0. 
			\end{align*}
	\end{enumerate}
\end{assumpt}

	Assumption \ref{Assumption_Sto_subgrad_basic}(a)-(b) is the same as Assumption \ref{Assumption_framework}(a)-(b). Moreover, Assumption \ref{Assumption_Sto_subgrad_basic}(c) holds whenever both $f$ and $\M$ are definable, and $\D_f = \partial f$  \cite[Corollary 6.4]{davis2020stochastic}, hence it is mild in practice. In addition, Assumption \ref{Assumption_Sto_subgrad_basic}(d) characterizes the way in which $(d_{x, k}, d_{y,k})$ 
	is an approximated evaluation of $\D_f(x_k, y_k - w_k)$ in the sense of conditional expectation.

	\begin{prop}
		\label{Prop_MDS}
		Suppose $\{\chi_k\}$ is a series of random variables such that for each $k \geq 1$, $\chi_k$ is $\ca{F}_{k+1}$-measurable, $\bb{E}[\chi_k |\ca{F}_k] = 0$, $\bb{E}[|\chi_k|] < +\infty$,  and $\sup_{k >1}\bb{E}[\norm{\chi_k}^2] < +\infty$. Then for any $\{\eta_k\}$ satisfying \eqref{Eq_Assumption_Sto}, $\sum_{j = 1}^k \eta_j \chi_j$ converges to a finite limit almost surely.  
	\end{prop}
	\begin{proof}
		Let $\tau_k := \sum_{j = 1}^k \eta_j \chi_j$. From the definition of $\tau_k$, we can conclude that  for each $k \geq 1$, $\bb{E}[ \tau_k |\ca{F}_k ] = \tau_{k-1}$ and $\bb{E}[|\tau_k|] \leq \sum_{j = 1}^k \eta_j\bb{E}[|\chi_j|] < +\infty$.
		Then $\{\tau_k\}$ is a martingale with respect to the filtration $\{\ca{F}_k\}$ \cite[Definition 5.1.4]{dembo2010probability}.  Moreover, since $\bb{E}[\chi_k |\ca{F}_k] = 0$, it holds that 
		\begin{equation*}
			\begin{aligned}
				\bb{E}[\norm{\tau_k}^2] = \bb{E} [\bb{E}[\norm{\tau_{k-1} + \eta_k \chi_k}^2 |\ca{F}_k]] \leq \eta_k^2\bb{E}[\norm{\chi_k}^2] + \bb{E}[\norm{\tau_{k-1}}^2],
			\end{aligned}
		\end{equation*} 
		Therefore, $\sup_{k >1} \bb{E}[\norm{\tau_k}^2] < +\infty$, hence $\{\tau_k\}$ is an $L^2$-martingale.   Then from \cite[Theorem 5.3.33]{dembo2010probability}, we can conclude that $\tau_k$ converges to a finite limit almost surely.  
	\end{proof}

	\begin{theo}
		\label{Theo_Alg_basic_sto}
		Suppose Assumption \ref{Assumption_Sto_subgrad_basic} holds,  $\beta \geq \frac{2M_fQ_g}{\mu^3}$ and the tolerances $\varepsilon_{1,k}$ and $\varepsilon_{2,k}$ satisfy $\lim\limits_{k\to +\infty} \varepsilon_{1,k} = 0$ and  $\sum\limits_{k = 0}^{+\infty} \varepsilon_{2,k} \etak < + \infty. $
		Then almost surely, every limit point of $\{(\xk, \yk)\}$ in Algorithm \ref{Alg:subgradient_method} is a $\ca{D}_f$-stationary point of \ref{Prob_Ori} and $\{h(\xk, y_k)\}$ converges. 
	\end{theo}
	\begin{proof}
		Consider the following auxiliary set-valued mapping $\ca{D}_{temp}: \Rn\times \Rp\times \Rp \rightrightarrows \Rn\times \Rp$,
		\begin{equation*}
			\ca{D}_{temp}(x, y, z):= \left\{ 
			\left[\begin{smallmatrix}
				d_x - \nabla_{xy}^2g(x, y) \left(\nabla_{yy}^2g(x, y)^{-1} d_y - \beta \nabla_y g(x, y)\right)\\
				\beta \nabla_y g(x, y)
			\end{smallmatrix}\right]  :    \left[\begin{smallmatrix}
				&d_x \\
				&d_y \\
			\end{smallmatrix}\right] \in \ca{D}_f(x, z) \right\}.
		\end{equation*}
		It is easy to verify that $\ca{D}_{temp}$ has closed graph. Moreover, $\hDh(x, y) =\ca{D}_{temp}(x, y, \A(x, y))$ holds for any $(x, y) \in \Rn\times \Rp$. 
		
		Assumption \ref{Assumption_Sto_subgrad_basic}(a) and \ref{Assumption_Sto_subgrad_basic}(b) imply that  Assumption \ref{Assumption_framework}(a) and \ref{Assumption_framework}(b) hold. 
		Let $\tilde{d}_{x, k} = \bb{E}[d_{x, k} |\ca{F}_k]$, $\tilde{d}_{y, k} = \bb{E}[d_{y, k} |\ca{F}_k]$, and 
		\begin{align*}
			u_{x,k} ={}& \tilde{d}_{x,k} -   \nabla_{xy}^2 g(\xk,\yk) \left( \nabla_{yy}^2 g(\xk, \yk)^{-1} \tilde{d}_{y,k} - \beta \nabla_y g(\xk,\yk)\right), ~~ u_{y,k} = \beta \nabla_y g(\xk, \yk),\\
			\chi_{x, k} ={}& (d_{x, k} - \tilde{d}_{x, k}) - \nabla_{xy}^2 g(\xk,\yk)\nabla_{yy}^2 g(\xk, \yk)^{-1}( d_{y, k} - \tilde{d}_{y,k}), \\
			\xi_{x,k} ={}&  	\nabla_{xy}^2 g(\xk,\yk) (\nabla_{yy}^2 g(\xk, \yk)^{-1} d_{y,k}	-v_k) + \chi_{x, k}, \quad \xi_{y,k} = 0. 
		\end{align*}
		Then from  Step 4 in Algorithm \ref{Alg:subgradient_method} we obtain $\norm{\xi_{x,k}} \leq \mu^{-1} L_g \varepsilon_{2,k}$. 
		As a result, Assumption \ref{Assumption_Sto_subgrad_basic}(b) shows that $\frac{L_g}{\mu} \sum_{k = 0}^{+\infty} \varepsilon_{2,k} \etak < +\infty$. Moreover, Proposition \ref{Prop_MDS} illustrates that $\sum_{k = 1}^{+\infty} \eta_k\chi_{x, k} < +\infty$.  Thus $\sum_{k = 0}^{+\infty} \norm{\eta_k\xi_{x,k}} < +\infty$  and Assumption \ref{Assumption_framework}(c) holds. 
		
		Assumption \ref{Assumption_Sto_subgrad_basic}(d) illustrates that $ \lim\limits_{k\to +\infty} \mathrm{dist} \left((u_{x,k}, u_{y,k}) , \ca{D}_{temp}(\xk,\yk,  \yk - w_k)\right) = 0$, and Step 2 in Algorithm \ref{Alg:subgradient_method} shows that $\lim\limits_{k\to +\infty} \norm{w_k - (\nabla_{yy}^2g(\xk, \yk))^{-1} \nabla_y g(\xk, \yk)} = 0$. For any sequence $\{k_j\} \subset \bb{N}$ such that $\lim_{j\to +\infty} (x_{k_j}, y_{k_j}) = (\tilde{x}, \tilde{y})$, it holds from Step 2 in Algorithm \ref{Alg:subgradient_method} that $ \lim_{j\to +\infty} (y_{k_j} - w_{k_j}) = \A(\tilde{x}, \tilde{y})$. 
		Then Proposition \ref{Prop_closed_graph_convex_approx} illustrates that 
		\begin{equation*}
			\small
			\lim_{N\to +\infty}\mathrm{dist}\left( \frac{1}{N}\sum_{j = 1}^{N}\left[\begin{matrix}
				u_{x, k_j} \\
				u_{y, k_j}
			\end{matrix}\right] , \hDh(\tilde{x}, \tilde{y})\right) = 
			\lim_{N\to +\infty}\mathrm{dist}\left( \frac{1}{N}\sum_{j = 1}^{N}\left[\begin{matrix}
				u_{x, k_j} \\
				u_{y, k_j}
			\end{matrix}\right] , \ca{D}_{temp}(\tilde{x}, \tilde{y}, \A(\tilde{x}, \tilde{y}))\right) = 0,
		\end{equation*}
		which guarantees Assumption \ref{Assumption_framework}(d). 
		
		Furthermore, Assumption \ref{Assumption_framework}(e) directly follows 
		from Assumption \ref{Assumption_Sto_subgrad_basic}(c) and Proposition \ref{Prop_field_approx}, and Assumption \ref{Assumption_framework}(f) is implied by Proposition \ref{Prop_descrease_approx}. Therefore,  Assumption \ref{Assumption_framework} holds for Algorithm \ref{Alg:subgradient_method}. 
		
		As a result, based on Theorem \ref{Theo_framework_convergence} and Theorem \ref{Theo_conservative_equivalence}, we obtain that  any cluster point of the sequence $\{(\xk, \yk)\}$ generated by Algorithm \ref{Alg:subgradient_method} is a $\ca{D}_f$-stationary point of \ref{Prob_Ori}, and the sequence $\{h(\xk, \yk)\}$ converges. 
	\end{proof}

	As illustrated in Remark \ref{Rmk_cf_subdifferential}, when we choose the $(d_{x,k}, d_{y,k})$ in Algorithm \ref{Alg:subgradient_method} by random sampling of gradient approaches, i.e., $(d_{x,k}, d_{y,k}) \in \partial_{\eta_k} f(\xk, \yk - w_k)$.  
	It holds from Proposition \ref{Prop_Eps_Subdifferential} that $\mathrm{dist}\big((d_{x,k}, d_{y,k}), \partial f(\xk, \yk - w_k)  \big) \to 0$. Hence $(d_{x,k}, d_{y,k})$ satisfies Assumption \ref{Assumption_Sto_subgrad_basic}(d) with $\D_f = \partial f$. 
	
	Similarly, if we choose $(d_{x,k}, d_{y,k})$ by randomized smoothing approaches, \cite[Theorem 3.1, Lemma D.1]{lin2022gradient} illustrates that Assumption \ref{Assumption_Sto_subgrad_basic}(d) is satisfied with $\D_f = \partial f$.  Then we immediately have the following corollary illustrating that $\{ (\xk, \yk) \}$ in Algorithm \ref{Alg:subgradient_method} converges to a first-order stationary point of \ref{Prob_Ori}. 
	\begin{coro}
		\label{Coro_basic_sgd_sampling}
		Suppose  Assumption \ref{Assumption_Sto_subgrad_basic} holds with $\D_f = \partial f$, the tolerances $\varepsilon_{1,k}$ and $\varepsilon_{2,k}$ satisfy $\lim\limits_{k\to +\infty} \varepsilon_{1,k} = 0$ and  $\sum\limits_{k = 0}^{+\infty} \varepsilon_{2,k} \etak < + \infty. $ Moreover, suppose $(d_{x, k}, d_{y, k})$ in Algorithm \ref{Alg:subgradient_method} is generated by one of the following schemes in each iteration $k$,
		\begin{itemize}
			\item $(d_{x,k}, d_{y,k}) \in \partial_{\eta_k} f(\xk, \yk-w_k)$;
			\item $(d_{x,k}, d_{y,k}) = \tilde{\partial}_{\eta_k} f(\xk, \yk-w_k; \zeta_{x, k}, \zeta_{y, k})$, where  $(\zeta_{x, k}, \zeta_{y, k})$ is uniformly sampled over $\bb{B}_{ {\delta} }(0)$ and  independent of $\ca{F}_k$. 
		\end{itemize}
		Then every limit point of $\{(\xk, \yk)\}$ in Algorithm \ref{Alg:subgradient_method} is a first-order stationary point of \ref{Prob_Ori} and $\{h(\xk, y_k)\}$ converges.  
		
	\end{coro}

\subsection{A modified subgradient method}
Recently, an efficient single-loop algorithm, named TTSA, is proposed by \cite{hong2020two} for \ref{Prob_Ori} with smooth $f$. The deterministic version of TTSA  follows the following updating schemes, 
\begin{equation}
	\label{Eq_updating_scheme_TTSA}
	\begin{aligned}
		\xkp ={}& \xk - \eta_k \left( \nabla_x f(\xk, \yk) -   \nabla_{xy}^2 g(\xk,\yk) \nabla_{yy}^2 g(\xk, \yk)^{-1} \nabla_y f(\xk, \yk)  \right),\\
		\ykp ={}& \yk - \tau_k \nabla_{y} g(\xk, \yk).
	\end{aligned}
\end{equation}
The $x$-variable in TTSA  is updated along an approximate gradient direction of $\Phi(x)$, while the $y$-variable is updated by taking a gradient descent step for the lower-level problem of \ref{Prob_Ori}. \cite{hong2020two} proves the global convergence of TTSA under a two-timescale condition, i.e., the ratio of stepsizes $\eta_k/ \tau_k$ tends to zero as the maximum number of iterations goes to infinity.  Very recently, \cite{khanduri2021near} proposes another single-loop algorithm named SUSTAIN, which can be regarded as a momentum-accelerated version of TTSA and waives the two-timescale condition in TTSA. However, the analysis for TTSA  and SUSTAIN  is based on the Lipschitz smoothness of $f$. To our best knowledge, the methodologies employed in \cite{hong2020two,khanduri2021near} cannot be applied to the nonsmooth bilevel problem \eqref{Prob_Ori}. 

In this subsection, we first consider the following set-valued mapping with a prefixed constant $\hat{\beta}>0$,
\begin{equation*}
	\hDs(x, y) := \left\{\left[ \begin{smallmatrix}
		&d_x - \nabla_{xy}^2g(x, y) \nabla_{yy}^2g(x, y)^{-1} d_y  \\
		&\hat{\beta}  \nabla_y g(x, y)
	\end{smallmatrix} \right]:    \left[\begin{smallmatrix}
		&d_x \\
		&d_y \\
	\end{smallmatrix}\right] \in \ca{D}_f(x, \A(x, y))
	\right\},
\end{equation*}
which yields a subgradient method as presented in Algorithm \ref{Alg:subgradient_SUSTAIN}.  Moreover, based on our proposed framework in Section \ref{Subsection_framework}, we prove the convergence properties of Algorithm \ref{Alg:subgradient_SUSTAIN} and discuss its relationship with the TTSA algorithm in Remark \ref{Remark_relationship_with_TTSA}. 

In the next two propositions, we establish some properties of $\hDs$.
\begin{prop}
	\label{Prop_field_sustain}
	Suppose $\beta \geq \frac{2Q_gM_f}{\mu^3}$. Then for any given $(x, y) \in \Rn \times \Rp$, $0 \in \Dh(x,y)$ if and only if $0 \in \hDs(x, y)$. 
\end{prop}
The proof is similar to Proposition \ref{Prop_field_approx}, hence we omit its proof for simplicity. 

\begin{prop}
	\label{Prop_descrease_approx_sustain}
	Suppose $\beta \geq \frac{4Q_gM_f}{\mu^3}$, and $\hat{\beta} \geq \beta 
	\max\big\{\frac{8L_g^2 }{\mu}, \frac{1}{4\mu}, \frac{\mu}{4}\big\}$. Then for any given $(x, y) \in \Rn \times \Rp$, and for any $w \in \hDs(x, y)$, it holds that 
	\begin{equation*}
		\sup_{z \in \Dh(x, y)} \inner{ \xi, z } \geq  \min\left\{ \frac{1}{4},   \frac{ \beta^2 }{16\hat{\beta}^2} \right\} \norm{w}^2. 
	\end{equation*}
\end{prop}
\begin{proof}
	For any $(d_x, d_y) \in \ca{D}_f(x, \A(x, y))$, let 
	\begin{equation}
		\label{Eq_Prop_descrease_approx_sustain_0}
		\begin{aligned}
			&z_1 = \left[ \begin{smallmatrix}
				&d_x - \nabla_{xy}^2g(x, y) \nabla_{yy}^2g(x, y)^{-1} d_y  \\
				&\beta \nabla_{yy}^2g(x, y) \nabla_y g(x, y)
			\end{smallmatrix} \right], \quad 
			z_2 = \left[ \begin{smallmatrix}
				&  \beta\nabla_{xy}^2g(x, y)   \nabla_y g(x, y)\\
				& 0
			\end{smallmatrix} \right],\\
			&z_3 = 
			\left[ \begin{smallmatrix}
				& \nabla_{xyy}^3 g(x, y)[\nabla_{yy}^2g(x, y)^{-1}\nabla_y g(x, y)] \nabla_{yy}^2g(x, y)^{-1} d_y \\
				&  \nabla_{yyy}^3 g(x, y)[\nabla_{yy}^2g(x, y)^{-1}\nabla_y g(x, y)] \nabla_{yy}^2g(x, y)^{-1} d_y
			\end{smallmatrix} \right], \quad w = \left[ \begin{smallmatrix}
				&d_x - \nabla_{xy}^2g(x, y) \nabla_{yy}^2g(x, y)^{-1} d_y  \\
				&\hat{\beta}  \nabla_y g(x, y)
			\end{smallmatrix} \right]. 
		\end{aligned}
	\end{equation}
	Then it holds that $ z_1 + z_2+z_3 \in \Dh(x, y)$ and $w \in \hDs(x, y)$. 
	From the expression of $w$ and the Lipschitz continuity of $\nabla_{yy}^2 g(x, y)$, we have 
	\begin{equation*}
		\norm{w} \leq \hat{\beta} \norm{\nabla_y g(x, y)} + \norm{d_x - \nabla_{xy}^2g(x, y) \nabla_{yy}^2g(x, y)^{-1} d_y},  \norm{z_3} \leq \frac{Q_g M_f}{\mu^2} \norm{\nabla_y g(x, y)},
	\end{equation*}
	which further implies that
	\begin{equation*}
		\begin{aligned}
			&\inner{z_3, w}\\
			\geq{}& - \frac{Q_g M_f}{\mu^2} \norm{\nabla_y g(x, y)}\norm{d_x - \nabla_{xy}^2g(x, y) \nabla_{yy}^2g(x, y)^{-1} d_y } - \frac{\hat{\beta}Q_g M_f}{\mu^2}  \norm{\nabla_y g(x, y)}^2\\
			\geq{}& -\frac{1}{4} \norm{d_x - \nabla_{xy}^2g(x, y) \nabla_{yy}^2g(x, y)^{-1} d_y }^2 - \left( \frac{\hat{\beta}Q_g M_f}{\mu^2} + \frac{Q_g^2 M_f^2}{\mu^4} \right) \norm{\nabla_y g(x, y)}^2.
		\end{aligned}
	\end{equation*}		
	As a result, we obtain
	\begin{align*}
		&\inner{w, z_1 + z_2 + z_3} 
		\geq \norm{d_x - \nabla_{xy}^2g(x, y) \nabla_{yy}^2g(x, y)^{-1} d_y }^2 + \beta\hat{\beta} \mu \norm{\nabla_y g(x, y)}^2 \\
		& - 2L_g\beta \norm{d_x - \nabla_{xy}^2g(x, y) \nabla_{yy}^2g(x, y)^{-1} d_y }\norm{\nabla_y g(x, y)}+ \inner{z_3, w} \\
		\geq{}& \frac{1}{2} \norm{d_x - \nabla_{xy}^2g(x, y) \nabla_{yy}^2g(x, y)^{-1} d_y }^2 + \left(\beta \hat{\beta} \mu - 2L_g^2\beta^2\right) \norm{\nabla_y g(x, y)}^2 + \inner{z_3, w}\\
		\geq{}& \frac{1}{4} \norm{d_x - \nabla_{xy}^2g(x, y) \nabla_{yy}^2g(x, y)^{-1} d_y }^2 + \left(\beta \hat{\beta} \mu - 2L_g^2\beta^2 - \frac{\hat{\beta}Q_g M_f}{\mu^2} - \frac{Q_g^2 M_f^2}{\mu^4} \right) \norm{\nabla_y g(x, y)}^2\\
		\geq{}& \frac{1}{4} \norm{d_x - \nabla_{xy}^2g(x, y) \nabla_{yy}^2g(x, y)^{-1} d_y }^2 + \frac{\beta^2}{16} \norm{\nabla_y g(x, y)}^2 \geq  \min\Big\{ \frac{1}{4},   \frac{ \beta^2 }{16\hat{\beta}^2}\Big\} \norm{w}^2,
	\end{align*}
	and  the proof is completed. 
\end{proof}

With Propositions \ref{Prop_descrease_approx_sustain} and \eqref{Eq_Prop_descrease_approx_sustain_0}, we are now ready to present our modified
subgradient method for solving \ref{Prob_Pen} in Algorithm 2, and establish its convergence.
\begin{algorithm}[htbp]
	\begin{algorithmic}[1]   
		\Require  Function $f$, $g$, initial point $x_0$, $y_0$.
		\For{k = 1,2,...}
		\State Set the tolerance $\varepsilon_{1,k}$ and $\varepsilon_{2,k}$.
		\State Compute $\nabla_y g(\xk, \yk)$. 
		\State Compute an approximated evaluation $w_{k}$ for $\nabla_{yy}^2g(\xk, \yk)^{-1} \nabla_y g(\xk, \yk)$ that satisfies $
		\norm{\nabla_{yy}^2 g(\xk, \yk) w_k - \nabla_y g(\xk, \yk) } \leq \varepsilon_{1,k}$. 
		\State Choose $(d_{x,k}, d_{y,k})$ as an approximated evaluation of $\ca{D}_f(\xk, \yk - w_k)$.
		\State Compute $v_k$ such that $
		\norm{\nabla_{yy}^2 g(\xk, \yk) v_k - d_{y,k}} \leq \varepsilon_{2,k}$.
		\State Update $x_k$ and $y_k$ by 
		\begin{align*}
			&\xkp =  \xk - \eta_k \left( d_{x,k} -   \nabla_{xy}^2 g(\xk,\yk) v_k  \right), \\
			&\ykp =  \yk - \eta_k \hat{\beta} \nabla_y g(\xk,\yk).
		\end{align*}
		\EndFor
		\State Return $x_k$ and $y_k$.
	\end{algorithmic}  
	\caption{A modified subgradient method for solving \ref{Prob_Pen}.}
	\label{Alg:subgradient_SUSTAIN}
\end{algorithm}
\begin{theo}
	Suppose Assumption \ref{Assumption_Sto_subgrad_basic} holds,  $\beta \geq \frac{4Q_gM_f}{\mu^3}$, $\hat{\beta} \geq \beta \cdot \max\big\{\frac{8L_g^2 }{\mu}, \frac{1}{4\mu},  \frac{\mu }{4}\big\}$ and the tolerance $\varepsilon_{1,k}$ and $\varepsilon_{2,k}$ satisfy $
	\lim\limits_{k\to +\infty} \varepsilon_{1,k} = 0$, $\sum\limits_{k = 0}^{+\infty} \varepsilon_{2,k} \etak < + \infty$. 
	Then  every limit point of $\{(\xk, \yk)\}$ generated by Algorithm \ref{Alg:subgradient_SUSTAIN} is a $\ca{D}_f$-stationary point of \ref{Prob_Ori} and $\{h(\xk, y_k)\}$ converges. 
\end{theo}
\begin{proof}
	Consider the auxiliary set-valued mapping $\ca{D}_{temp}: \Rn\times \Rp\times \Rp \rightrightarrows \Rn\times \Rp$ that is defined as
	\begin{equation*}
		\ca{D}_{temp}(x, y, z):= \left\{ 
		\left[\begin{smallmatrix}
			d_x - \nabla_{xy}^2g(x, y) \nabla_{yy}^2g(x, y)^{-1} d_y\\
			\hat{\beta} \nabla_y g(x, y)
		\end{smallmatrix}\right]  :    \left[\begin{smallmatrix}
			&d_x \\
			&d_y \\
		\end{smallmatrix}\right] \in \ca{D}_f(x, z) \right\}.
	\end{equation*}
	Then it is easy to verify that $\ca{D}_{temp}$ has closed graph. 
	
	Assumption \ref{Assumption_Sto_subgrad_basic}(a) and \ref{Assumption_Sto_subgrad_basic}(b) implies that  Assumption \ref{Assumption_framework}(a) and \ref{Assumption_framework}(b) hold. 
		Let $\tilde{d}_{x, k} = \bb{E}[d_{x, k} |\ca{F}_k]$, $\tilde{d}_{y, k} = \bb{E}[d_{y, k} |\ca{F}_k]$, and 
		\begin{align*}
			u_{x,k} ={}& \tilde{d}_{x,k} -   \nabla_{xy}^2 g(\xk,\yk)  \nabla_{yy}^2 g(\xk, \yk)^{-1} \tilde{d}_{y,k} , ~~u_{y,k} = \hat{\beta} \nabla_y g(\xk, \yk),\\
			\chi_{x, k} ={}& (d_{x, k} - \tilde{d}_{x, k}) - \nabla_{xy}^2 g(\xk,\yk)\nabla_{yy}^2 g(\xk, \yk)^{-1}( d_{y, k} - \tilde{d}_{y,k}), \\
			\xi_{x,k} ={}&  	\nabla_{xy}^2 g(\xk,\yk) (\nabla_{yy}^2 g(\xk, \yk)^{-1} d_{y,k}	-v_k) + \chi_{x, k}, \quad \xi_{y,k} = 0. 
		\end{align*}
		Then from  Step 6 in Algorithm \ref{Alg:subgradient_SUSTAIN} we obtain $\norm{\xi_{x,k}} \leq   \mu^{-1} L_g \varepsilon_{2,k}$. 
		As a result, Assumption \ref{Assumption_framework}(b) shows that $  \mu^{-1} L_g \sum_{k = 1}^{+\infty} \varepsilon_{2,k} \etak < +\infty$. Furthermore, from Proposition \ref{Prop_MDS} it holds that $\sum_{k = 1}^{+\infty} \chi_{x, k}$ converges to a finite limit almost surely. Therefore, $\sum_{k = 0}^{+\infty} \eta_k\xi_{x,k}  $ converges to a finite limit almost surely, and hence Assumption \ref{Assumption_framework}(c) holds. 
		
		Notice that $\mathrm{dist} \big((u_{x,k}, u_{y,k}), \ca{D}_{temp}(\xk, \yk, \yk - w_k)\big) \to 0$. Moreover, Step 4 in Algorithm \ref{Alg:subgradient_method} shows that $\lim\limits_{k\to +\infty} \norm{w_k - \nabla_{yy}^2g(\xk, \yk) \nabla_y g(\xk, \yk)} = 0$. For any sequence $\{k_j\} \subset \bb{N}$ such that $\lim\limits_{j\to +\infty} (x_{k_j}, y_{k_j}) = (\tilde{x}, \tilde{y})$, $ y_{k_j} - w_{k_j} \to \A(x_{k_j}, y_{k_j})$. 
		Then Proposition \ref{Prop_closed_graph_convex_approx} illustrates that 
		\begin{equation*}
			\small
			\lim_{N\to +\infty}\mathrm{dist}\left( \frac{1}{N}\sum_{j = 1}^{N}\left[\begin{matrix}
				u_{x, k_j} \\
				u_{y, k_j}
			\end{matrix}\right] , \hDs(\tilde{x}, \tilde{y})\right) = 
			\lim_{N\to +\infty}\mathrm{dist}\left( \frac{1}{N}\sum_{j = 1}^{N}\left[\begin{matrix}
				u_{x, k_j} \\
				u_{y, k_j}
			\end{matrix}\right] , \ca{D}_{temp}(\tilde{x}, \tilde{y}, \A(\tilde{x}, \tilde{y}))\right) = 0,
		\end{equation*}
		which guarantees Assumption \ref{Assumption_framework}(d). 
	
	Furthermore, Assumption \ref{Assumption_framework}(e) directly follows from Assumption \ref{Assumption_Sto_subgrad_basic}(c) and Proposition \ref{Prop_field_sustain}, and Assumption \ref{Assumption_framework}(f) is implied by Proposition \ref{Prop_descrease_approx_sustain}. From Theorem \ref{Theo_framework_convergence} and Theorem \ref{Theo_conservative_equivalence}, we can conclude that for the sequence $\{(\xk, \yk)\}$ generated by Algorithm \ref{Alg:subgradient_SUSTAIN}, any cluster point of $\{(\xk, \yk)\}$ is a $\ca{D}_f$-stationary point of \ref{Prob_Ori}, and the sequence $\{h(\xk,\yk)\}$ converges.
\end{proof}

	Similar to Corollary \ref{Coro_basic_sgd_sampling}, the following corollary illustrates that when $(d_{x, k}, d_{y, k})$ in Step 5 of Algorithm \ref{Alg:subgradient_SUSTAIN} is generated by the randomized approaches mentioned in Remark \ref{Rmk_cf_subdifferential}, the yielded sequence $\{(\xk, \yk)\}$ converges to a first-order stationary point of \ref{Prob_Ori}. 
	\begin{coro}
		\label{Coro_TTSA_sampling}
		Suppose  Assumption \ref{Assumption_Sto_subgrad_basic} holds with $\D_f = \partial f$,  $\beta \geq \frac{4Q_gM_f}{\mu^3}$, $\hat{\beta} \geq \beta \cdot \max\big\{\frac{8L_g^2 }{\mu}, \frac{1}{4\mu},  \frac{\mu }{4}\big\}$,  and the tolerance $\varepsilon_{1,k}$ and $\varepsilon_{2,k}$ satisfy $
		\lim\limits_{k\to +\infty} \varepsilon_{1,k} = 0$, $\sum\limits_{k = 0}^{+\infty} \varepsilon_{2,k} \etak < + \infty$. Moreover, suppose $(d_{x, k}, d_{y, k})$ in Algorithm \ref{Alg:subgradient_SUSTAIN} is generated by one of the following schemes in each iteration $k$,
		\begin{itemize}
			\item $(d_{x,k}, d_{y,k}) \in \partial_{\eta_k} f(\xk, \yk-w_k)$;
			\item $(d_{x,k}, d_{y,k}) = \tilde{\partial}_{\eta_k} f(\xk, \yk-w_k; \zeta_{x, k}, \zeta_{y, k})$, where  $(\zeta_{x, k}, \zeta_{y, k})$ is uniformly sampled over $\bb{B}_{ {\delta} }(0)$ and  independent of $\ca{F}_k$.
		\end{itemize}
		Then almost surely,  every limit point of $\{(\xk, \yk)\}$ generated by Algorithm \ref{Alg:subgradient_SUSTAIN} is a first-order stationary point of \ref{Prob_Ori} and $\{h(\xk, y_k)\}$ converges. 
		
\end{coro}

\begin{rmk}
	\label{Remark_relationship_with_TTSA}
	When $f$ is assumed to be Lipschitz smooth over $\Rn \times \Rp$, 
	Algorithm \ref{Alg:subgradient_SUSTAIN} coincides with \eqref{Eq_updating_scheme_TTSA}, which can be regarded as the deterministic version of the TTSA algorithm in \cite{hong2020two}, and the SUSTAIN algorithm with $\eta_t^g = \eta_t^f = 1$ in \cite[Equation (13)-(14)]{khanduri2021near} (i.e. SUSTAIN algorithm without momentum accelerations). Therefore, the deterministic version of TTSA  can be interpreted as an approximated gradient descent algorithm that minimizes \ref{Prob_Pen} over $\Rn\times \Rp$, while  SUSTAIN  can be regarded as a momentum-accelerated (stochastic) gradient method for solving \eqref{Prob_Pen}.  Moreover, as illustrated in Algorithm \ref{Alg:subgradient_SUSTAIN}, we can extend the deterministic version of these algorithms to handle nonsmooth bilevel optimization problems based on our proposed framework. 
\end{rmk}

\subsection{An inexact subgradient method}
Recently, another efficient single-loop approach named STABLE \cite{chen2021single}, is proposed for  nonconvex-strongly-convex bilevel optimization problems where the objective functions are assumed to be Lipschitz smooth over $\Rn\times \Rp$.  The deterministic version of STABLE algorithm employs the following updating schemes,
\begin{equation}
	\label{Eq_updating_scheme_STABLE}
	\begin{aligned}
		& \xkp = \xk - \eta_k \left(\nabla_x f(\xk,\yk) -    \nabla_{xy}^2 g(\xk,\yk) \nabla_{yy}^2 g(\xk, \yk)^{-1} \nabla_y f(\xk,\yk) \right), \\
		&\ykp = \yk - \tau_k \nabla_y g(\xk, \yk)   +\eta_k \nabla_{yx}^2 g(\xk, \yk) \nabla_{yy}^2g(\xk, \yk)^{-1}\Big(\nabla_x f(\xk,\yk)  \\
		&\qquad\qquad   -    \nabla_{xy}^2 g(\xk,\yk) \nabla_{yy}^2 g(\xk, \yk)^{-1} \nabla_y f(\xk,\yk) \Big).  
	\end{aligned}
\end{equation}
Here the $x$-variable takes an approximated gradient descent step for $\Phi(x)$. However, the updating schemes of $y$-variable can be hard to understand by regarding STABLE algorithm as an approximated gradient descent algorithm for minimizing $\Phi(x)$.

In this subsection, we propose an inexact subgradient method based on our proposed framework with the following set-valued mapping $\hDp(x,y)$,
\begin{equation}
	\hDp(x, y) = W(x, y)\tp W(x, y) \ca{D}_f(x, \A(x, y)) + 
	\left[\begin{smallmatrix}
		0\\
		\beta \nabla_y g(x, y)\\
	\end{smallmatrix} \right],
\end{equation}
where $W(x,y) \in \bb{R}^{n\times (n+p)}$ is defined by $W(x,y) = \left[ I_n, -\nabla_{xy}^2g(x, y) \nabla_{yy}^2g(x, y)^{-1} \right]$.

We first prove that $\hDp(x,y)$ has compact and convex values, and satisfies the Assumption \ref{Assumption_framework}(f). Moreover, based on $\hDp(x,y)$, we propose a subgradient method as presented in Algorithm \ref{Alg:STABLE_Deterministic} and show its global convergence properties directly from our proposed framework. A discussion on how to understand STABLE algorithm based on \ref{Prob_Pen} is presented at the end of this subsection.

In the next two propositions, we establish some properties of $ \hDp(x, y)$.
\begin{prop}
	\label{Prop_Subgradient_Equivalence_STABLE_1}
	Suppose $\beta \geq \frac{2M_fQ_g}{\mu^3}$. Then for any $(x,y) \in \Rn\times \Rp$, $0 \in \Dh(x,y)$ if and only if $	0 \in \hDp(x, y)$. 
\end{prop}
\begin{proof}
	When $0 \in \Dh(x,y)$, it holds from Theorem \ref{Theo_conservative_equivalence} that $(x,y) \in \M$ and there exists $ (d_x, d_y) \in \ca{D}_f(x, y)$ such that $0 =  d_x - \nabla_{xy}^2g(x, y) \nabla_{yy}^2g(x, y)^{-1}d_y \in W(x,y)\ca{D}_f(x,y)$. Therefore, we conclude that 
	\begin{equation*}
		0 \in W(x, y)\tp W(x, y) \ca{D}_f(x, y) + 
		\left[\begin{smallmatrix}
			0\\
			\beta \nabla_y g(x, y)\\
		\end{smallmatrix} \right] = \hDp(x,y). 
	\end{equation*}				
	On the other hand, suppose $0 \in \hDp(x, y)$, then there exists $(d_x, d_y) \in \ca{D}_f(x, \A(x, y))$ such that 
	\begin{align*}
		&d_x - \nabla_{xy}^2g(x, y) \nabla_{yy}^2g(x, y)^{-1}d_y = 0 \label{Eq_Prop_Subgradient_Equivalence_STABLE_1}\\
		& -\nabla_{yy}^2g(x, y)^{-1}\nabla_{yx} g(x, y)  \left(d_x - \nabla_{xy}^2g(x, y) \nabla_{yy}^2g(x, y)^{-1}d_y\right) + \beta \nabla_y g(x, y) = 0. 
	\end{align*}
	As a result, it holds that $\nabla_y g(x, y) = 0$ and hence $(x, y) \in \M$. Together with \eqref{Eq_updating_scheme_STABLE} and Definition \ref{Defin_FOSP_Ori_field}, we obtain that $(x, y)$ is a $\Dh$-stationary point of \ref{Prob_Ori}.
\end{proof}

\begin{prop}
	\label{Prop_STABLE_descent}
	Suppose $\beta \geq \max\left\{ \frac{8M_fQ_g}{\mu^{3}}, \frac{4M_fQ_gL_g}{\mu^{3.5}}\right\}$. Then for any $(x, y) \in \Rn\times \Rp$ and any $w \in \hDp(x, y)$, it holds that
	\begin{equation*}
		\sup_{z \in \Dh(x, y)} \inner{z, w} \geq \min\left\{ \frac{\mu^2}{4L_g^2}, \frac{\mu}{4} \right\} \norm{w_1 + w_2}^2.
	\end{equation*}
\end{prop}
\begin{proof}
	For any $(d_x, d_y) \in \ca{D}_f(x, \A(x, y))$, let $z_1$, $z_2$ and $z_3$ be defined as
	\begin{equation*}
		\begin{aligned}
			&z_1 = \left[\begin{smallmatrix}
				&d_x -\nabla_{xy}^2g(x, y) \nabla_{yy}^2g(x, y)^{-1} d_y \\
				& 0 \\
			\end{smallmatrix}\right],\quad z_2 = \left[ \begin{smallmatrix}
				& \nabla_{xyy}^3 g(x, y)[\nabla_{yy}^2g(x, y)^{-1}\nabla_y g(x, y)] \nabla_{yy}^2g(x, y)^{-1} d_y \\
				&  \nabla_{yyy}^3 g(x, y)[\nabla_{yy}^2g(x, y)^{-1}\nabla_y g(x, y)] \nabla_{yy}^2g(x, y)^{-1} d_y
			\end{smallmatrix} \right],\\
			&z_3 = \left[\begin{smallmatrix}
				& \beta \nabla_{xy}^2 g(x, y) \nabla_y g(x, y) \\
				& \beta \nabla_{yy}^2 g(x, y) \nabla_y g(x, y) \\
			\end{smallmatrix}\right], \quad w_1 = W(x,y)\tp  W(x,y)
			\left[
			\begin{smallmatrix}
				d_x\\
				d_y\\
			\end{smallmatrix}
			\right], \quad w_2 = \left[\begin{smallmatrix}
				0\\
				\beta \nabla_y g(x, y)\\
			\end{smallmatrix} \right].
		\end{aligned}
	\end{equation*}
	Then it holds that $z_1 + z_2 + z_3 \in \Dh(x, y)$, and $w_1 + w_2 \in \hDp(x, y)$.  
	Moreover, as $\norm{z_2} \leq \frac{2M_f Q_g }{\mu^2} \norm{\nabla_y g(x, y)}$, we obtain the following inequalities through simple calculations, 
	\begin{align*}
		& \inner{z_1, w_1} = \norm{z_1}^2 , ~ \inner{z_2, w_1} \geq -  \left(\frac{2M_f Q_g }{\mu^2}\right) \norm{\nabla_y g(x, y)}\norm{w_1}, ~\inner{z_3, w_1} = 0,\\
		& \inner{z_1, w_2} = 0, ~ 
		\inner{z_2, w_2} \geq - \frac{2M_f Q_g\beta }{\mu^2} \norm{\nabla_y g(x, y)}^2, ~ \inner{z_3, w_2}\geq \mu\beta^2 \norm{\nabla_y g(x, y)}^2.
	\end{align*}
	By Cauchy's inequality, it holds from $\beta \geq \frac{4M_fQ_gL_g}{\mu^{3.5}}$ that 
	\begin{equation*}
		\frac{\mu^2}{4L_g^2 } \norm{w_1}^2 + \frac{\mu\beta^2}{4} \norm{\nabla_y g(x, y)}^2 \geq  \frac{2M_f Q_g }{\mu^2} \norm{\nabla_y g(x, y)}\norm{w_1}.
	\end{equation*}
	Therefore, we get
	\begin{equation*}
		\begin{aligned}
			&\inner{z_1+ z_2 + z_3, w_1 + w_2}\\
			\geq{}& \norm{z_1}^2 + \mu\beta^2 \norm{\nabla_y g(x, y)}^2 - \left(\frac{2M_f Q_g }{\mu^2}\right) \norm{\nabla_y g(x, y)}\norm{w_1} 
			- \frac{2M_f Q_g\beta }{\mu^2} \norm{\nabla_y g(x, y)}^2\\ 
			\geq{}& \frac{\mu^2}{2L_g^2  }\norm{w_1}^2 
			+ \mu\beta^2 \norm{\nabla_y g(x, y)}^2 - \left(\frac{2M_f Q_g }{\mu^2}\right) \norm{\nabla_y g(x, y)}\norm{w_1} - \frac{\mu\beta^2}{4}\norm{\nabla_y g(x,y)}^2\\ 
			\geq{}& \frac{\mu^2}{4L_g^2 } \norm{w_1}^2 + \frac{\mu\beta^2}{4} \norm{\nabla_y g(x, y)}^2
			\geq \min\left\{ \frac{\mu^2}{4L_g^2}, \frac{\mu}{4} \right\} \norm{w_1 + w_2}^2,
		\end{aligned}
	\end{equation*}
	and this completes the proof. 
\end{proof}

With Propositions \ref{Prop_Subgradient_Equivalence_STABLE_1} and \ref{Prop_STABLE_descent}, we can now present our inexact subgradient method for solving 
\ref{Prob_Pen} in Algorithm 3 and establish its convergence.
\begin{algorithm}[htbp]
	\begin{algorithmic}[1]   
		\Require  Function $f$, $g$, initial point $x_0$, $y_0$.
		\For{k = 1,2,...}
		\State Compute $w_k$ by approximately evaluating $\nabla_{yy}^2 g(\xk, \yk)^{-1} \nabla_y g(\xk, \yk)$ such that $
		\norm{\nabla_{yy}^2 g(\xk, \yk) w_k - \nabla_y g(\xk, \yk) } \leq \varepsilon_{1,k} $.
		\State Choose $(d_{x,k}, d_{y,k})$ as an approximated evaluation of $\ca{D}_f(\xk, \yk - w_k) $.
		\State Compute $p_{x,k} = d_{x,k} -    \nabla_{xy}^2 g(\xk,\yk) \nabla_{yy}^2 g(\xk, \yk)^{-1} d_{y,k}$.
		\State Update $\xk$ and $\yk$  by 
		\begin{align*}
			&\xkp =  \xk - \eta_k p_{x,k},\\
			& \ykp = \yk - \etak\left(\beta \nabla_y g(\xk, \yk) -\nabla_{yx}^2 g(\xk, \yk) \nabla_{yy}^2g(\xk, \yk)^{-1}p_{x,k}\right). 
		\end{align*}
		\EndFor
		\State Return $x_k$ and $y_k$.
	\end{algorithmic}  
	\caption{Inexact subgradient method for solving \ref{Prob_Pen}.}
	\label{Alg:STABLE_Deterministic}
\end{algorithm}

\begin{theo}
	Suppose Assumption \ref{Assumption_Sto_subgrad_basic} holds, $\beta \geq \max\left\{\frac{8M_fQ_g}{\mu^{3}}, \frac{4M_fQ_gL_g}{\mu^{3.5}}\right\}$ and $\lim_{k\to +\infty} \varepsilon_{1,k} = 0$. 
	Then every limit point of $\{(\xk, \yk)\}$ generated by Algorithm \ref{Alg:STABLE_Deterministic} is a $\ca{D}_f$-stationary points of \ref{Prob_Ori} and $\{h(\xk, y_k)\}$ converges. 
\end{theo}
\begin{proof}
		Assumption \ref{Assumption_Sto_subgrad_basic}(a) and \ref{Assumption_Sto_subgrad_basic}(b) imply that  Assumption \ref{Assumption_framework}(a) and \ref{Assumption_framework}(b) hold. 
		Let $
		\hat{\ca{D}}_{temp}(x, y, z) := W(x, y)\tp W(x, y) \ca{D}_f(x, z) + 
		\left[\begin{smallmatrix}
			0\\
			\beta \nabla_y g(x, y)\\
		\end{smallmatrix} \right]$, $\tilde{d}_{x, k} = \bb{E}[d_{x, k} |\ca{F}_k]$, $\tilde{d}_{y, k} = \bb{E}[d_{y, k} |\ca{F}_k]$, and  
		\begin{align*}
			u_{x,k} ={}& \tilde{d}_{x,k} -   \nabla_{xy}^2 g(\xk,\yk) \nabla_{yy}^2 g(\xk, \yk)^{-1} \tilde{d}_{y,k} ,\\
			\hat{\chi}_{x,k} ={}& (d_{x,k} - \tilde{d}_{x,k}) -   \nabla_{xy}^2 g(\xk,\yk) \nabla_{yy}^2 g(\xk, \yk)^{-1}  (d_{y,k} -\tilde{d}_{y,k}) ,\\
			u_{y,k} ={}&  \beta \nabla_y g(\xk, \yk) -   \nabla_{yx}^2 g(\xk,\yk) \nabla_{yy}^2 g(\xk, \yk)^{-1} u_{x, k}, \\
			\xi_{x,k} ={}& - \nabla_{yx}^2 g(\xk,\yk) \nabla_{yy}^2 g(\xk, \yk)^{-1} \hat{\chi}_{x,k}, \quad \xi_{y,k} = 0. 
		\end{align*}
		It is easy to verify the validity of Assumption \ref{Assumption_framework}(c) from Proposition \ref{Prop_MDS}. Moreover, $\hat{\ca{D}}_{temp}(x, y, \A(x, y)) = \hDp(x, y)$ holds for any $(x,y) \in \Rn\times \Rp$, and 
		\begin{equation*}
			\lim_{k\to +\infty} \mathrm{dist}\left((u_{x,k}, u_{y,k}) , \hat{\ca{D}}_{temp}(\xk, \yk, \yk - w_k)\right) = 0.
		\end{equation*} 
		Furthermore, notice that $\lim_{k\to +\infty} \varepsilon_{1,k} = 0$. Then for any subsequence $\{ (x_{k_j}, y_{k_j}) \}$ that converges to $\{(\tilde{x},\tilde{y})\}$, it holds that $(\xk, \yk, \yk-w_k)$ converges to $(\tilde{x}, \tilde{y}, \A(\tilde{x}, \tilde{y}))$. Then  Proposition \ref{Prop_closed_graph_convex_approx} illustrates that 
		$
		\lim\limits_{N\to +\infty}\mathrm{dist}\left( \frac{1}{N}\sum_{j = 1}^{N}\left[\begin{matrix}
			u_{x, k_j} \\
			u_{y, k_j}
		\end{matrix}\right] , \hDp(\tilde{x}, \tilde{y})\right)  = 0$, 
		which verifies the validity of Assumption \ref{Assumption_framework}(d). Additionally,  Assumption \ref{Assumption_Sto_subgrad_basic}(c) implies Assumption \ref{Assumption_framework}(e), and Proposition \ref{Prop_STABLE_descent} guarantees the validity of Assumption \ref{Assumption_framework}(f). Then from Theorem \ref{Theo_framework_convergence}, we can conclude that $\{h(\xk, \yk)\}$ converges and any cluster point of $\{(\xk, \yk)\}$ yielded by Algorithm \ref{Alg:STABLE_Deterministic} is a $\ca{D}_f$-stationary point of \ref{Prob_Ori}. 
	\end{proof}
	
	Similar to Corollary \ref{Coro_basic_sgd_sampling} and Corollary \ref{Coro_TTSA_sampling}, we have the following corollary illustrating that $\{(\xk, \yk)\}$ weakly converges to first-order stationary points of \ref{Prob_Ori} when $(d_{x, k}, d_{y,k})$ is generated by some randomized approaches mentioned in Remark \ref{Rmk_cf_subdifferential}.

	\begin{coro}
		Suppose Assumption \ref{Assumption_Sto_subgrad_basic} holds with $\D_f = \partial f$,  $\lim_{k\to +\infty} \varepsilon_{1,k} = 0$, and $\beta \geq \max\left\{\frac{8M_fQ_g}{\mu^{3}}, \frac{4M_fQ_gL_g}{\mu^{3.5}}\right\}$. Moreover, suppose $(d_{x, k}, d_{y, k})$ in Algorithm \ref{Alg:subgradient_SUSTAIN} is generated by one of the following schemes in each iteration $k$,
		\begin{itemize}
			\item $(d_{x,k}, d_{y,k}) \in \partial_{\eta_k} f(\xk, \yk-w_k)$;
			\item $(d_{x,k}, d_{y,k}) = \tilde{\partial}_{\eta_k} f(\xk, \yk-w_k; \zeta_{x, k}, \zeta_{y, k})$, where  $(\zeta_{x, k}, \zeta_{y, k})$ is uniformly sampled over $\bb{B}_{ {\delta} }(0)$ and  independent of $\ca{F}_k$. 
		\end{itemize}
		Then every limit point of $\{(\xk, \yk)\}$ generated by Algorithm \ref{Alg:STABLE_Deterministic} is a first-order stationary point of \ref{Prob_Ori} and $\{h(\xk, y_k)\}$ converges. 
	\end{coro}

\begin{rmk}
	When $f$ is assumed to be Lipschitz smooth over $\Rn \times \Rp$, Algorithm \ref{Alg:STABLE_Deterministic} coincides with the updating schemes \eqref{Eq_updating_scheme_STABLE} of the deterministic version of the STABLE algorithm. As illustrated in Proposition \ref{Prop_STABLE_descent},  the deterministic version of STABLE  can be regarded as a descent algorithm for $h$ in \ref{Prob_Pen}  in each iteration. This provides a clear understanding  
	of the convergence properties of the STABLE algorithm, and demonstrates the efficiency of Algorithm \ref{Alg:STABLE_Deterministic}. Moreover, according to Step 4 in Algorithm \ref{Alg:STABLE_Deterministic}, the stepsizes $\eta_k$ and $\tau_k$ in \eqref{Eq_updating_scheme_STABLE} should satisfy $\tau_k = \beta \eta_k$, which further explains the different theoretical bounds for  $\eta_k$ and $\tau_k$ suggested in  \cite[Theorem 2]{chen2021single}. 
	Therefore, we can conclude that  \ref{Prob_Pen} exhibits its ability in interpreting 
	the STABLE algorithm and allows great flexibility
	in employing advanced theoretical analysis developed for unconstrained optimization.  
\end{rmk}

\section{Conclusion}

In this paper, we propose an unconstrained optimization problem \ref{Prob_Pen} for the bilevel optimization problem \ref{Prob_Ori}.  We prove that under mild conditions, \ref{Prob_Ori} and \ref{Prob_Pen} have the same stationary points over $\Rn\times \Rp$ in the sense of both Clarke subdifferential and conservative field.  Moreover, \ref{Prob_Pen} has explicit formulation, and its  function value and corresponding conservative field can be easily calculated in the presence of  $\ca{D}_f$ and the derivatives of $g$. Therefore, various prior arts for unconstrained nonsmooth optimization can be directly employed to solve \ref{Prob_Ori} through the unconstrained optimization problem \ref{Prob_Pen}.

We propose a unified framework for developing subgradient methods, which further inspires several subgradient-based methods for solving \ref{Prob_Ori} through \ref{Prob_Pen}. In addition, we show that the proposed framework provides simple interpretations for some existing single-loop algorithms. Specifically, we show that the TTSA, SUSTAIN and STABLE algorithm can be regarded as approximated first-order methods for minimizing \ref{Prob_Pen} when $f$ is assumed to be Lipschitz smooth. Based on our proposed framework, we can straightforwardly extend these algorithms to nonsmooth cases and establish their global convergence properties.

Furthermore, suppose the objective functions $f$ and $g$ in \ref{Prob_Ori} are expressed as the expectation of some random variables, i.e. 
	\begin{equation*}
		f(x, y) = \bb{E}_{\xi}[ f_{\xi}(x,y) ], \quad g(x, y) = \bb{E}_{\theta}[ g_{\theta}(x,y) ],
	\end{equation*}
	where $f_{\xi}: \Rn\times \Rp \to \bb{R}$ and $g_{\theta}: \Rn\times \Rp \to \bb{R}$ are continuous functions that depend on the random variables $\xi$ and $\theta$, respectively. Then  the corresponding \ref{Prob_Pen} can be formulated as
	\begin{equation*}
		\footnotesize
		\begin{aligned}
			\min_{x \in \Rn, y \in \Rp}~ \tilde{h}(x, y):=& 
			\bb{E}_{\xi}\left[ f_{\xi}\left(x, y-\bb{E}_{\theta}\left[ \nabla_{yy}^2 g_{\theta}(x,y) \right]^{-1} \bb{E}_{\theta}\left[ \nabla_{y} g_{\theta}(x,y) \right] \right) \right]  + \frac{\beta}{2} \norm{\bb{E}_{\theta}[ \nabla_{y} g_{\theta}(x,y) ]}^2,
		\end{aligned}
	\end{equation*}
	which can be categorized as a special case of unconstrained conditional stochastic optimization \cite{hu2020sample}. Therefore,  we can directly apply some  existing advanced approaches \cite{lian2017finite,hu2020biased,chen2021solving,gao2021fast} to solve \ref{Prob_Pen} when $f$ is Lipschitz smooth. Moreover, their theoretical properties, including global convergence, iteration complexity and sample complexity, directly follow the results from these existing works. We leave the discussion on how to design efficient algorithms to minimize $\tilde{h}$ over $\Rn\times \Rp$ for future investigation. 
	
	\bibliography{ref}
	\bibliographystyle{plainnat}

\end{document}